\documentclass[a4paper]{article}
\usepackage[plainpages=false]{hyperref}
\usepackage{amsfonts,latexsym,rawfonts,amsmath,amssymb,amsthm, mathrsfs, lscape}
\usepackage{verbatim}

\usepackage[all]{xy}
\usepackage{graphicx,psfrag}

\usepackage{array,tabularx}

\usepackage{setspace}

\newtheorem{thm}{Theorem}[section]
\newtheorem{cor}{Corollary}[section]
\newtheorem{lem}[thm]{Lemma}

\newtheorem{prop}[thm]{Proposition}

\theoremstyle{remark}
\newtheorem{rmk}[thm]{Remark}

\theoremstyle{definition}
\newtheorem{defi}{Definition}[section]                                     

\numberwithin{equation}{section}

\def\p{\partial}
\def\R{\mathbb{R}}
\def\C{\mathbb{C}}
\def\N{\mathbb{N}}

\def\Q{\mathbb{Q}}
\def \T {\mathbb T}

\def\l{\lambda}

\def\i{\sqrt{-1}}
	
\def\t{\Delta}

\def\I {\mathbb{I}}
\def \ft {\mathfrak t}
\def\o{\omega}

\def\g{{\mathfrak g}}

\def\J{{\mathcal J}}
\def\cA{{\mathcal A}}

\def\cD{\mathcal D}

\def\cF{{\mathcal F}}
\def\cG{{\mathcal G}}

\def\cH{{\mathcal H}}

\def\cJ{{\mathcal J}}

\def\cL{{\mathcal L}}
\def\cM{{\mathcal M}}

\def\cO{{\mathcal O}}

\def\cR{{\mathcal R}}

\def\cU{{\mathcal U}}
\def\cV{{\mathcal V}}
\def\cW{{\mathcal W}}

\def \bp {\overline{\partial}}

\def \Ker {\text{Ker}}
\def\Aut{\text{Aut}}
\def \Diff{\text{Diff}}
\def \id{\text{id}}
\def\Vol{\text{Vol}}
\def \Diff {\text{Diff}}

\begin{document}

\title{Frankel conjecture and Sasaki geometry}

\author{Weiyong He\footnote{Partially supported by an NSF grant, award No. DMS-1005392.}  \ and Song Sun\footnote{Partially supported by European Research Council award No 247331. }}

\maketitle

\begin{abstract} We classify simply connected  compact Sasaki manifolds of dimension $2n+1$ with positive transverse bisectional curvature. In particular, the K\"ahler cone corresponding to such manifolds must be bi-holomorphic to $\C^{n+1}\backslash \{0\}$. As an application we  recover the  theorem of Mori and Siu-Yau on the Frankel conjecture and extend it to certain orbifold version. The main idea is to deform such Sasaki manifolds to the standard round sphere in two steps, both fixing the complex structure on the K\"ahler cone.  First, we deform the metric along the Sasaki-Ricci flow and obtain a limit  Sasaki-Ricci soliton with positive transverse bisectional curvature.  Then  by varying the Reeb vector field which essentially decreases the volume functional, we deform  the Sasaki-Ricci soliton to a Sasaki-Einstein metric with positive transverse bisectional curvature, i.e. a round sphere. The second deformation is only possible when one treats simultaneously regular and irregular Sasaki manifolds, even if the manifold one starts with is regular (quasi-regular),  i.e. K\"ahler manifolds (orbifolds).
\end{abstract}


\section{Introduction and main results}
In this paper we study compact Sasaki manifolds with positive transverse bisectional curvature. Sasaki geometry, in particular, Sasaki-Einstein manifolds have been studied extensively. Readers are referred to the monograph \cite{BG}, the survey paper \cite{Sparks} and the references therein for the history and recent progress on this subject.

The study of manifolds with positive curvature is one of the most important subjects in Riemannian geometry. There are lots of recent deep progress on this, especially using the technique of Ricci flow, see \cite{BW} and \cite{BS1} for example.  In K\"ahler geometry a natural concept is the positivity of the bisectional curvature. It was conjectured by Frankel \cite{Frankel61} that a compact K\"ahler manifold of complex dimension $n$ with positive bisectional curvature is biholomorphic to the complex projective space $\mathbb{CP}^n$. The Frankel conjecture was proved in later 1970s independently  by Mori \cite{Mori} (he proved the more general Hartshorne conjecture) via algebraic geometry and Siu-Yau \cite{Siu-Yau} via differential geometry.  Sasaki geometry is an odd dimensional companion of K\"ahler geometry, so it is vary natural to ask for the counterpart of the theorem of Mori and Siu-Yau on the Frankel conjecture for Sasaki manifolds. This is the major point of study in this article. We would like to emphasize that this generalization seems to be interesting, in that it provides a uniform framework  which also proves the original Frankel conjecture, by deformation to canonical metrics, as attempted previously by many people(c.f. \cite{Chen-Tian1, Chen-Tian2},  \cite{PSSW}).  Moreover, the use of Sasaki geometry also yields  certain orbifold version of the Frankel conjecture, which seems to be difficult to obtain with the known approaches.  Finally, as already pointed out in \cite{BW}  a \emph{pinching towards constant curvature} proof of the Frankel conjecture using Ricci flow seems not plausible, as there are examples of two dimensional Ricci soliton orbifolds with positive curvature. One of the applications of the results developed in this article is to classify such solitons, in a uniform way. 

Sasaki geometry in dimension $2n+1$  is closely related to K\"ahler geometry in both dimensions $2(n+1)$ and $2n$.  A Sasaki manifold $M$ of dimension $2n+1$ admits, on one hand,  a K\"ahler cone structure on the product $X=M\times\R_{+}$, and on the other hand, a transverse K\"ahler structure on the (local) quotient by the Reeb vector field. For now  we view a Sasaki structure on $M$ as a K\"ahler cone structure on $X$, and we identify $M$ with the link $\{r=1\}$  in $X$. A standard example of a Sasaki manifold is the odd dimensional round sphere $S^{2n+1}$. The corresponding K\"ahler cone is $\C^{n+1}\backslash\{0\}$ with the flat metric.   

A Sasaki manifold admits a  canonical Killing vector field $\xi$, called the \emph{Reeb vector field}. It is given by rotating the homothetic vector field $r\p_r$ on $X$ by the complex structure $J$.  The integral curves of $\xi$ are geodesics, and give rise  to a foliation on $M$, called the \emph{Reeb foliation}. Then there is a K\"ahler structure on the local leaf space of the Reeb foliation, called the \emph{transverse K\"ahler structure}. If the transverse K\"ahler structure has positive bisectional curvature, we say the Sasaki manifold has \emph{positive transverse bisectional curvature}. If the Sasaki manifold has positive sectional curvature, it automatically has positive transverse bisectional curvature, for example, the round metric on $S^{2n+1}$. Indeed for  the round sphere $S^{2n+1}$,  the transverse K\"ahler structure on the leaf space is isometric to the Fubini-Study metric on $\mathbb{CP}^n$. \\

The main goal of this article is to classify compact Sasaki manifolds with positive transverse bisectional curvature. By a homothetic transformation such manifolds always admit Riemannian metrics with positive Ricci curvature, so they  must have finite fundamental group. Therefore without loss of generality, we may assume the manifolds are simply connected. Our main result is

\begin{thm}\label{T-main}
Let $(M, g)$ be a compact simply connected Sasaki manifold of dimension $2n+1$ with positive transverse bisectional curvature, then its K\"ahler cone $(X,  J)$ is biholomorphic to $\C^{n+1}\backslash\{0\}$. Moreover, $M$ is a \emph{weighted Sasaki sphere}, i.e. $M$ is diffeomorphic to the sphere $S^{2n+1}$ and the Sasaki metric is a \emph{simple Sasaki metric} on $S^{2n+1}$.
\end{thm}

Roughly speaking, a \emph{simple Sasaki metric} on $S^{2n+1}$
is a Sasaki metric that can be deformed to the round metric on $S^{2n+1}$ through a \emph{simple deformation}. 
The relevant definitions will be given in Section \ref{S-2}. When $n=1$, our proof implies that any Sasaki structure on $S^3$ is simple and its K\"ahler cone is $\C^{2}\backslash\{0\}$; this result was proved by Belgun \cite{Belgun} as a part of the classification of three dimensional Sasaki manifolds. In a sequel \cite{He-Sun}, we will use the results of this paper together with the technique of Brendle-Schoen \cite{BS} (c.f. also \cite{Gu}) to classify compact Sasaki manifolds with non-negative transverse bisectional curvature. 

As a direct consequence of Theorem \ref{T-main}, we obtain the classification of compact \emph{polarized orbifolds} with positive bisectional curvature.

\begin{cor}\label{orbifold Frankel}
A compact polarized orbifold $(M, J, g,  L)$ with positive bisectional curvature is bi-holomorphic to a finite quotient of a weighted projective space. 
\end{cor}

The notion \emph{polarized orbifold} is taken from \cite{RT}. By this we mean there is an orbi-line bundle $L$, and  in any orbifold chart $(U_p, L_p, G_p)$ the action of $G_p$ on $L_p$ is faithful. 
As a special case of Corollary \ref{orbifold Frankel},  we obtain an alternative analytic  proof of  Siu-Yau's Theorem.
\begin{cor}[ \cite{Mori}, \cite{Siu-Yau}]\label{Frankel conjecture}
A compact K\"ahler manifold with positive bisectional curvature is bi-holomorphic to the complex projective space. 
\end{cor}

One interesting point here is that our proof of Corollary \ref{orbifold Frankel} and \ref{Frankel conjecture} do  rely on the framework of Sasaki geometry. 
A converse of Theorem \ref{T-main} is also true. 

\begin{thm}\label{T-main-2}
Any simple Sasaki structure on $S^{2n+1}$ can be deformed to a Sasaki-Ricci soliton with positive transverse bisectional curvature, by a transverse K\"ahler deformation. In particular,  a weighted projective space  carries an  orbifold K\"ahler-Ricci soliton with positive bisectional curvature.
\end{thm}

The existence of Sasaki-Ricci solitons on weighted Sasaki sphere follows from the result of Futaki-Ono-Wang \cite{FOW} on toric Sasaki manifolds. We will prove that these Sasaki-Ricci solitons all have positive transverse bisectional curvature. We remark that here these metrics are not explicit, and  we are not able to find a general way of producing an explicit orbifold  K\"ahler metric with positive bisectional curvature even on weighted projective spaces. \\

Before we sketch the main ideas to prove Theorem \ref{T-main}, it is valuable to recall the known proofs of the Frankel conjecture. In both \cite{Mori} and \cite{Siu-Yau}
 the existence of rational curves  plays an essential role.  Mori proved a more general result that a Fano manifold always  contains a rational curve by a bend-and-break argument and the algebraic geometry in positive characteristic; while Siu-Yau used the Sacks-Uhlenbeck argument to produce a stable harmonic sphere, and exploited the positivity of bisectional curvature to prove that such a sphere is either holomorphic or anti-holomorphic. A key ingredient in the proof of Siu-Yau is a characterization of the complex projective space by Kobayashi-Ochiai \cite{KO}.  There is, to the authors' knowledge so far, no  analogue of this in Sasaki geometry to characterize a weighted Sasaki sphere, or particularly a weighted projective space. This seems to be a major obstacle for adapting the approach of Siu-Yau to the Sasaki case.
 
We proceed along a different track in this paper, that is, by  deforming a geometric structure naturally to a standard one that can be classified more.   
In early 1980s Hamilton \cite{Hamilton82} introduced the \emph{Ricci flow},  as a powerful tool  to evolve Riemannian metrics towards canonical models.  On K\"ahler manifolds the Ricci flow preserves the K\"ahler condition. It is called the \emph{K\"ahler-Ricci flow} and was first studied by Cao \cite{Cao}. Bando \cite{Bando} (for complex dimension three) and Mok \cite{Mok} (for all dimensions) studied the K\"ahler-Ricci flow on compact manifolds with positive(non-negative) bisectional curvature. They
proved that  this positivity (non-negativity) is preserved along the flow, using Hamilton's maximum principle for tensors.  Since then, there have been many attempts  to seek a proof of the Frankel conjecture using Ricci flow, and there has been extensive study of K\"ahler-Ricci flow with positive (nonnegative) bisectional curvature. We mention  \cite{Chen-Tian1, Chen-Tian2, CCZ, PSSW} to name a few.  Note that Berger \cite{Berger} proved that a K\"ahler-Einstein metric with positive sectional curvature is isometric to the complex projective space with the Fubini-Study metric; this result was later generalized to K\"ahler-Einstein manifolds with positive bisectional curvature by Goldberg-Kobayashi \cite{GK} where they first introduced the concept of holomorphic bisectional curvature.   One can get an alternative proof of the Mori's and Siu-Yau's theorem on the Frankel conjecture if  the K\"ahler-Ricci flow converges to a K\"ahler-Einstein metric with positive bisectional curvature.
 
Later on Perelman introduced many revolutionary ideas, including the by now well-known entropy functionals \cite{Perelman01} into the study of the Ricci flow, which lead him to the solution of the Poincar\'e conjecture and Thurston's geometrization conjecture. He also proved very deep results for the  K\"ahler-Ricci flow on Fano manifolds,  namely, that the scalar curvature and the diameter are uniformly bounded along the flow; details of his results can be found in \cite{Sesum-Tian}. Combining this with Mok's results,  it then easily follows that the K\"ahler-Ricci flow converges by sequence  to a K\"ahler-Ricci soliton up to diffeomorphisms, if the initial metric has positive (nonnegative) bisectional curvature.  Using the Morse-Bott theory and dimension induction, Chen, Tian and the second author \cite{CST}  gave a direct proof that the limit K\"ahler-Ricci soliton, and hence the original K\"ahler manifold is biholomorphic to the complex projective space. The proof  still depends on producing rational curves and  applying the results of Kobayashi-Ochiai.  Along the Ricci flow we only know that a K\"ahler-Ricci soliton with positive bisectional curvature must be K\"ahler-Einstein \emph{a posteriori}, and a direct proof of this is still lacking. 

Given the analogue between them, many concepts, techniques and results can be carried over from K\"ahler geometry to Sasaki geometry with certain modifications.  Sasaki-Ricci flow was introduced by Smoczyk-Wang-Zhang \cite{SWZ} as a counterpart of K\"ahler-Ricci flow; it  deforms a Sasaki metric such that its transverse K\"ahler metric is deformed by the transverse K\"aher-Ricci flow.  Independently, Collins \cite{Collins} and the first author \cite{He} generalized Perelman's entropy and corresponding results in the K\"ahler-Ricci flow on Fano manifolds to the Sasaki setting by considering only \emph{basic geometric data}, i.e. geometric quantities that are invariant along the Reeb vector fields. Then the $W$ functional is  monotone along Sasaki-Ricci flow, and the (transverse) scalar curvature and  the diameter are both uniformly bounded along the flow. Furthermore, the first author studied the Sasaki-Ricci flow with positive (nonnegative) transverse bisectional curvature. It is shown that the flow  converges to a Sasaki-Ricci soliton with positive transverse bisectional curvature. It is also proved in \cite{He} that a compact, simply connected Sasaki-Einstein manifold with positive transverse bisectional curvature has constant transverse holomorphic sectional curvature, hence is the round sphere.  

Here comes an essential difference from the K\"ahler setting. The fact that compact K\"ahler-Ricci solitons with positive bisectional curvature must be K\"ahler-Einstein might simply be a coincidence, since if we consider the general Sasaki setting, then a compact Sasaki-Ricci soliton  with positive transverse bisectional curvature does not have to be Sasaki-Einstein. Indeed L.F. Wu(\cite{Wu}, see also \cite{CW}) proved the existence of a non-trivial Ricci soliton with positive curvature on $S^2$ with certain orbifold singularity; this implies the existence of a non-Einstein Sasaki-Ricci soliton on $S^3$ with positive transverse curvature. Recently  Futaki-Ono-Wang \cite{FOW} proved the existence of Sasaki-Ricci solitons on compact toric Sasaki manifolds; their results produced  a family of toric Sasaki-Ricci solitons on the weighted Sasaki sphere $S^{2n+1}$; the positivity condition can be assured if a Sasaki-Ricci soliton is close to the round metric on the sphere.  Actually Theorem \ref{T-main-2}  asserts that all toric Sasaki-Ricci solitons in this family have positive transverse bisectional curvature. In short, the model structure in the Sasaki setting is not a unique one, but a whole family. 

The problem is now reduced to classifying Sasaki-Ricci solitons with positive transverse bisectional curvature. Our strategy is to deform such solitons to a Sasaki-Einstein metric with positive transverse bisectional curvature. Note that there is no such corresponding deformation within the framework of K\"ahler geometry. Such a flexibility in Sasaki setting seems to be one of the advantages of the new approach. This not only implies that a K\"ahler-Ricci soliton with positive bisectional curvature is K\"ahler-Einstein, hence gives a new analytic proof of Siu-Yau theorem, but also allows us to generalize the results to the Sasaki setting (Theorem \ref{T-main}). 

To carry out the deformation of Sasaki-Ricci solitons, we first recall the theory of volume minimization due to Martelli-Sparks-Yau \cite{MSY}. It is observed in \cite{MSY} that  the volume of a compact Sasaki manifold  is equivalent to the Einstein-Hilbert functional, and is a function of the Reeb vector field only. Fix the complex structure on the K\"ahler cone $(X, J)$ and a maximal compact torus $\T$ in the automorphism group $\Aut(X, J)$, they obtained a beautiful variational picture of the volume functional on the Lie algebra of $\T$. In particular, the functional is convex  and its critical point, if exists, is the Reeb vector field for the putative Sasaki-Einstein metric; moreover, the first variation of the volume functional can be interpreted  as the Futaki invariant (see also \cite{FOW}). It is then very natural to deform  Reeb vector fields to reduce the volume functional, and to deform Sasaki-Ricci solitons correspondingly,  with the hope to reach a critical Reeb vector field where we end up with a desired Sasaki-Einstein metric. This is the heuristic strategy we take. Along the way we also develop the rudiments for the theory of this new deformation, and we hope it will also be useful in more general setting (c.f. Section 6).  \\

Now we outline the organization of the article.  In Section \ref{S-2} we set up various definitions. In Section \ref{S-2-1} we give a gentle introduction to Sasaki geometry. In particular, we recall the notion of transverse bisectional curvature. In Section \ref{S-2-2} we introduce the notion of a \emph{Reeb cone} and a \emph{simple deformation} of Sasaki structures, for our purpose in this paper.  In Section \ref{S-2-5} we introduce  weighted Sasaki spheres and simple Sasaki structures on them, which form the canonical models for our study.  In Section \ref{S-2-3} we recall the notion of  a Sasaki-Ricci soliton. Section \ref{S-2-4} studies the normalization we use when we deform the Reeb vector fields.

The main geometric study is in Section \ref{S-3}.   In Section \ref{S-3-2} and \ref{S-3-3} we  study the volume functional and Perelman's $\mu$ functional. These provide \emph{a priori} geometric bounds for our deformation.  In Section \ref{S-3-4} we carry out the above deformation picture to prove Theorem \ref{T-main},  \ref{T-main-2}, and  Corollary \ref{orbifold Frankel},  \ref{Frankel conjecture}, modulo technical results proved in Section 4 and 5. 

In Section \ref{S-4} we study the local property of the deformation, using implicit function theorem to prove two technical results. Section \ref{S-4-1} is concerned with the local deformation of Sasaki-Ricci solitons, and in Section \ref{S-4-2} we prove the rigidity of Sasaki manifolds with positive transverse bisectional curvature. This rigidity is crucial here since in complex geometry one often meets the problem of \emph{jumping phenomenon}. 
In Section \ref{S-5} we study compactness of a sequence of  Sasaki-Ricci solitons with positive transverse bisectional curvature. In Section \ref{S-9} we discuss related problems.  \\

\textbf{Acknowledgements}: We would like to thank Prof. Xiuxiong Chen and Prof. Simon Donaldson for warm encouragements. We are also grateful to Hans-Joachim Hein for valuable discussions, and Prof. Dmitri Panov and Prof. Richard Thomas for their interest in this work. We would also like to thank the referee for various suggestions and comments that greatly improve the exposition of this paper and for pointing out some inaccuracy of the previous version.

\section{Preliminaries in Sasaki geometry}\label{S-2}

Sasaki geometry has many equivalent descriptions. We  will largely use the formulation by  K\"ahler cones; see, for example, \cite{MSY} for a nice reference. It can also be defined in terms of metric contact geometry or transverse K\"ahler geometry; see \cite{BG}, for references.  

\subsection{Sasaki manifolds}\label{S-2-1}
Let $M$ be a compact differentiable manifold of dimension $2n+1 (n\geq 1)$. A \emph{Sasaki structure} on $M$ is defined to be a K\"ahler cone structure on $X=M\times \R_{+}$, i.e. a K\"ahler metric $(g_X, J)$ on $X$  of the form $$g_X=dr^2+r^2g,$$ where $r>0$ is a coordinate on $\R_{+}$, and $g$ is a Riemannian metric on $M$. We call $(X, g_X, J)$ the \emph{K\"ahler cone} of $M$. The vertex is not viewed as part of the cone throughout this paper. 
We also identify $M$ with the link $\{r=1\}$ in $X$ if there is no ambiguity.  Because of the cone structure, the K\"ahler form on $X$ can be expressed as 
$$\omega_X=\frac{1}{2}\sqrt{-1}\p\bp r^2=\frac{1}{4}dd^c r^2.$$
We denote by $r\p_r$ the homothetic vector field on the cone, which is easily seen to be a real holomorphic vector field. 
A tensor $\alpha$ on $X$ is said to be of homothetic degree $k$ if 
$$\cL_{r\p_r} \alpha=k\alpha.$$
In particular, $\omega$ and $g$ have homothetic degree two, while $J$ and $r\p_r$ has homothetic degree zero.
We define the \emph{Reeb vector field} $$\xi=J(r\p_r).$$
Then $\xi$ is a holomorphic Killing field on $X$ with homothetic degree zero. Let $\eta$ be the dual one-form to $\xi$:
 \[\eta(\cdot)=r^{-2}g_X(\xi, \cdot)=d^c \log r=\i (\bp-\p)\log r\ .\] 
We also use $(\xi, \eta)$ to denote the restriction of them on $(M, g)$.  Then we have 
\begin{itemize}
\item $\eta$ is a contact form on $M$, and $\xi$ is a Killing vector field on $M$ which we also call the Reeb vector field;
\item $\eta(\xi)=1, \iota_{\xi} d\eta(\cdot)=d\eta (\xi, \cdot)=0$;
\item the integral curves of $\xi$ are geodesics.  
\end{itemize}

The Reeb vector field $\xi$ defines a foliation $\cF_\xi$ of $M$ by geodesics. There is a classification of Sasaki structures according to the global property of the leaves. If all the leaves are compact, then $\xi$ generates a circle action on $M$, and the Sasaki structure is called {\it quasi-regular}. In general this action is only locally free, and we get a polarized orbifold structure on the leaf space. If the circle action is globally free, then the Sasaki structure is called {\it regular}, and the leaf space is a polarized K\"ahler manifold. If $\xi$ has a non-compact leaf the Sasaki structure is called {\it irregular}.    Readers are referred to Section \ref{S-2-5} for examples. In the present paper the regularity of a Sasaki structure will not be essential. 

There is an orthogonal decomposition of the tangent bundle \[TM=L\xi\oplus \cD,\] where $L\xi$ is the trivial  bundle generalized by $\xi$, and $\cD=\Ker (\eta)$.
The metric $g$ and the contact form $\eta$ determine a $(1,1)$ tensor field $\Phi$ on $M$ by
\[
g(Y, Z)=\frac{1}{2} d\eta(Y, \Phi Z), Y, Z\in \Gamma(\cD). 
\]
$\Phi$ restricts to an almost complex structure on $\cD$: \[\Phi^2=-\mathbb{I}+\eta\otimes \xi. \] 

Since both $g$ and $\eta$ are invariant under $\xi$, there is a well-defined K\"ahler structure $(g^T, \omega^T, J^T)$ on the local leaf space of the Reeb foliation. We call this a \emph{transverse K\"ahler structure}. In the quasi-regular case, this is the same as the K\"ahler structure on the quotient. Clearly 
\[
\o^T=\frac{1}{2}d\eta. 
\]
The upper script $T$ is used to denote both the transverse geometric quantity, and the corresponding quantity on the bundle $\cD$. For example we have on $M$
$$g=\eta\otimes  \eta+g^T.$$
From the above discussion it is not hard to see that there is an intrinsic formulation of a Sasaki structure as a compatible integrable pair $(\eta, \Phi)$, where $\eta$ is a contact one form and $\Phi$ is a almost CR structure on $\mathcal D=\Ker \eta$. Here ``compatible" means  first that 
$d\eta(\Phi U, \Phi V)=d\eta(U, V)$ for any $U, V\in \mathcal D$, and $d\eta(U, \Phi U)>0$ for any non zero $U\in \mathcal D$. Further we require $\mathcal L_{\xi}\Phi=0$, where $\xi$ is the unique vector field with $\eta(\xi)=1$, and $d\eta(\xi, \cdot)=0$. 
$\Phi$ induces a splitting $$\mathcal  D\otimes \C=\mathcal D^{1,0}\oplus \mathcal D^{0,1}, $$
 with $\overline{\mathcal D^{1,0}}=\mathcal D^{0,1}$. 
``Integrable" means that $[\mathcal D^{0,1}, \mathcal D^{0,1}]\subset \mathcal D^{0,1}$. This is equivalent to that the induced almost complex structure on the local leaf space of the foliation by $\xi$ is integrable. For more discussions on this, see \cite{BG} Chapter 6. 

The Sasaki structure on $M$ is determined by the triple $(\xi, \eta, g)$. \footnote{Indeed, $\xi$ and $\eta$ are determined by each other, but we keep the notation here in order to emphasis both the Reeb vector field and the contact form} From now on, we will  use the notation $(M, \xi, \eta, g)$ to denote a Sasaki manifold. 
By an easy computation for any tangent vector $Y$, 
\begin{equation}\label{e-2-1}
R(Z, \xi)Y=g(\xi, Y)Z-g(Z, Y)\xi. 
\end{equation}
It follows  that the sectional curvature of any tangent plane in $M$ containing $\xi$ has to be $1$; or equivalently, the sectional curvature of any tangent plane in $X$ containing either $\p_r$ or $\xi$, is zero. Hence if a Sasaki manifold $(M, g)$ of dimension $2n+1$ is Einstein, then the Einstein constant must be $2n$, i.e.
\[
Ric=2n g;
\]
correspondingly, the K\"ahler cone $(X, g_X, J)$ is then a Ricci flat cone, i.e.
\[
Ric_X=Ric-2ng=0. 
\]
One can introduce the transverse connection $\nabla^T$ and transverse curvature operator $R^T(Y, Z)W$ for $Y, Z,$ $W\in\Gamma(\cD)$. 
A straightforward computation shows that
\[
Ric(Y, Z)=Ric^T(Y, Z)-2g^T(Y, Z), Y, Z\in \Gamma(\cD).
\]
Hence the Sasaki-Einstein equation can also be written as a transverse K\"ahler-Einstein equation:
\begin{equation}\label{e-2-2}
Ric^T-2(n+1) g^T=0. 
\end{equation}

We are  interested in transverse holomorphic bisectional curvature. It has been studied recently \cite{Zhang, He}.  We recall some  definitions. 
\begin{defi}Given two $\Phi$-invariant tangent planes $\sigma_1, \sigma_2$ in $\cD_x\subset T_xM$, the \emph{transverse holomorphic bisectional curvature} $H^T(\sigma_1, \sigma_2)$ is defined as
\[
H^T(\sigma_1, \sigma_2)=\langle R^T(Y, JY)JZ, Z\rangle,
\]
where $Y\in \sigma_1, Z\in \sigma_2$ are both of unit length.  We define the transverse holomorphic sectional curvature of a $\Phi$-invariant tangent plane as
$$H^T(\sigma)=H^T(\sigma, \sigma).$$
\end{defi}
It is easy to check these are well-defined. For brevity, we will simply say ``transverse bisectional curvature" instead of ``transverse holomorphic bisectional curvature".

\begin{defi}For $x\in M$,  we say the transverse bisectional curvature is positive (or nonnegative) at $x$ if $H^T(\sigma_1, \sigma_2)$ is positive for any two $\Phi$ invariant planes $\sigma_1, \sigma_2$ in $\cD_x$. We say $M$ has positive (nonnegative) transverse  bisectional curvature if $H^T$ is positive (nonnegative) at any point $x\in M$; 
\end{defi}

It is often convenient to introduce  transverse holomorphic coordinates. Let $(z^1, \cdots, z^n)$ be a  holomorphic chart on a local leaf space around $x$. Then the transverse bisectional curvature is positive at $x$ if and only if for any two non zero tangent vectors $u=\sum u^i\frac{\p}{\p z^i}$, and $v=\sum v^j\frac{\p}{\p z^j}$, 
$$\langle R^T(u, \bar{v})v, \bar{u}\rangle= R^T_{i\bar{j}k\bar{l} }u^i u^{\bar{j}} v^k v^{\bar{l}}>0.$$

The transverse bisectional curvature determines the transverse sectional curvature, so by (\ref{e-2-1}) it determines the sectional curvature of $M$. We have the following classification, which is the starting point of our study.  

\begin{lem} \label{SE positive}
A compact simply connected Sasaki-Einstein manifold with positive transverse bisectional curvature is isomorphic to   the standard Sasaki structure on $S^{2n+1}$, or equivalently, the K\"ahler cone is isometric to the standard flat cone $\C^{n+1}\setminus \{0\}$. 
\end{lem}

\begin{proof}
This is essentially a known result. In \cite{He}, using the maximum principle as in \cite{Chen-Tian1}, it is proved that such manifolds must have constant transverse holomorphic bisectional curvature $1$. Equivalently,  this Sasaki structure has constant $\Phi$-holomorphic sectional curvature $1$. 
S. Tanno \cite{Tanno1} has given a full classification of simply-connected Sasaki manifolds with constant $\Phi$-holomorphic sectional curvature. In particular, he proved that a
simply connected Sasaki manifolds with constant $\Phi$-holomorphic sectional curvature $1$ is isometric (as a Sasaki structure) to the standard Sasaki structure on $S^{2n+1}(1)$, see Proposition 4.1 in \cite{Tanno1}. While the Kahler cone corresponding to the standard Sasaki structure on $S^{2n+1}(1)$ is just the standard flat cone $\C^{n+1}\setminus \{0\}$.\end{proof}

\subsection{Deformation of Sasaki structures}\label{S-2-2} 
Let $(M, \xi, \eta, g)$ be a given Sasaki structure. Note that for any positive constant $\l\neq 1$, the naive scaling $(M, \l g)$ is not a Sasaki metric, for example, by \eqref{e-2-1}. But there is a well-known replacement in the Sasaki setting, called \emph{D-homothetic transformation}, that is introduced by S. Tanno \cite{Tanno} (we shall use homothetic transformation for simplicity). It is induced by the transformation  $\xi\mapsto\l^{-1} \xi$ and $\eta\mapsto \l \eta$; the corresponding metric is then given by
\[
g_\l=\l^2 \eta\otimes \eta+\l g^T. 
\] 
Hence the transverse K\"ahler metric is rescaled indeed, but the scaling factor is different from that along the Reeb vector field direction.  On the cone $X$,  this can be realized by the transformation $r\mapsto \tilde r=r^\l$, and the K\"ahler form is given by
\[
\tilde{\o}=\frac{\i}{2} \p\bp \tilde r^2=\frac{\i}{2} \p \bp r^{2\lambda}. 
\]
 In the present paper we will fix a particular scaling normalization, and it will be specified later (see \eqref{normalization}). 

The deformations of Sasaki structures on $M$ that we are interested in will all be induced by a deformation of the K\"ahler cone metrics on $X$, with a fixed complex structure $J$, i.e. a deformation of the K\"ahler potentials $r^2$.  We first consider \emph{transverse K\"ahler deformation}, as discussed in \cite{MSY, FOW}. This is  a special case of a \emph{Type II deformation} introduced in \cite{BGM}. 
Given a Sasaki structure $(M, \xi, \eta, g)$ and its K\"ahler cone $(X, g_X, J)$. 
We consider all K\"ahler cone metrics on $(X, J)$ with Reeb vector field $\xi$. This is equivalent to fixing the homothetic vector field $r\p_r=-J\xi$. Let $\tilde r^2/2$ be the K\"ahler potential of another such K\"ahler cone metric $\tilde{g}$. Then we have
\[
\tilde r\p_{\tilde r}=r\p_r. 
\] 
Writing $\tilde r^2=r^2 \exp(2\phi)$, the condition becomes $\p_r\phi=0$.
Note that the K\"ahler cone condition implies $\cL_\xi \tilde r=0$. It follows that
$\cL_{\xi} \phi=0.$
Hence $\phi$ can also be considered as  a \emph{basic function} on $M$. 
On the cone $(X, J)$  we have
\[
\tilde \eta=J\left(d \log \tilde r\right)=\eta+d^c\phi.
\]
We summarize the discussion above as follows,

\begin{defi} Let $(M, \xi, \eta, g)$ be a Sasaki structure and $(X, J)$ be the underlying complex manifold of its Kahler cone. A \emph{transverse Kahler transformation} is induced by a \emph{basic function} $\phi$ on $M$ such that $(X, J, \xi)$ remains unchanged and the Sasaki structure is induced by the contact form $\tilde \eta=\eta+d^c \phi$ and $\tilde r=r\exp(\phi)$. 
\end{defi}

It is also very natural to present this deformation in terms of \emph{basic forms} on the Sasaki manifold $M$; see \cite{FOW, BG}, for example, for nice references. First we recall,
\begin{defi} A $p$-form $\theta$ on $M$ is called basic if
\[
\iota_\xi \theta=0, L_\xi \theta=0.
\]
Let $\Lambda^p_B$ be the sheaf of germs of basic $p$-forms and $\Omega^p_B=\Gamma(S, \Lambda^p_B)$ the space of  smooth sections of $\Lambda^p_B$.  
\end{defi}
The exterior differential preserves basic forms and we set $d_B=d|_{\Omega^p_B}$. 
Thus the subalgebra $\Omega_{B}(\cF_\xi)$ forms a subcomplex of the de Rham complex, and its cohomology ring $H^{*}_{B}(\cF_\xi)$  is called the {\it basic cohomology ring}. In particular, there is a transverse Hodge theory \cite{EKAH86, KT87, Ton97}. 
The transverse Hodge star operator $*_{B}$ is defined in terms of the usual Hodge star by
\[
*_{B} \alpha=*(\eta\wedge \alpha). 
\]
The adjoint $d^{*}_B: \Omega^p_{B}\rightarrow \Omega^{p-1}_B$ of $d_B$ is
\[
d^*_{B}=-*_{B} d_B *_{B}. 
\]
The {\it basic Laplacian} operator is defined to be $\t_B=d_B d^{*}_B+d^{*}_Bd_B$. 
When $(M, \xi, \eta, g)$ is a Sasaki structure, there is a natural splitting of $\Lambda^p_B\otimes \C$ such that
\[
\Lambda^p_B\otimes \C=\oplus \Lambda^{i, j}_B,
\]
where $\Lambda^{i, j}_B$ is the bundle of type $(i, j)$ basic forms. We thus have the well-defined operators
\[
\begin{split}
\p_B: \Omega^{i, j}_B\rightarrow \Omega^{i+1, j}_B,\\
\bar\p_B: \Omega^{i, j}_B\rightarrow \Omega^{i, j+1}_B.
\end{split}
\]
Then we have $d_B=\p_B+\bar \p_B$. 
Set $d^c_B=\i\left(\bar \p_B-\p_B\right);$ then
\[
d_Bd_B^c=2\i\p_B\bar\p_B, d_B^2=(d_B^c)^2=0.
\]
The transverse K\"ahler form defines a basic cohomology class $[\o^T]_B$.

Now we return to the above deformation.
$\tilde\eta$ could be viewed as the contact one-form of a new Sasaki structure on $M$ by pulling back through the embedding of $M$ into $X=M\times \R_{+}$ as $\{\tilde r=1\}=\{r=e^{-\phi(x)}\}$.  It is straightforward to check that $\xi$,  $\eta$ and $d^c\phi$ are invariant under the diffeomorphism  $$F_{\phi}: X\rightarrow X; (x, r)\mapsto (x, re^{-\phi(x)}),$$
So as contact one-forms on $M$, we have
$$\tilde\eta=\eta+d^c_B\phi.$$
Therefore, the transverse K\"ahler forms are related by
\[
\tilde \o^T=\o^T+\i \p_B\bp_B \phi.
\] 
In the regular case, this corresponds to deform the K\"ahler metric on the quotient K\"ahler manifold within a fixed K\"ahler class.  The transverse Ricci form $\rho^T$ defines a basic cohomology class $\frac{1}{2\pi}[\rho^T]_B$, which we call \emph{the basic first Chern class} $c_1^B$. 

Suppose now $(M, \xi, \eta, g)$ has positive transverse bisectional curvature, then it follows directly $c_1^B>0$. 
Moreover, a rather standard Bochner technique implies $b_2^B=\text{dim} H^2_B(M, \R)=1$ (see \cite{He} Section 9). It then follows that there is a positive constant $\l$ such that
\[
c_1^B=\frac{1}{2\pi}[\rho^T]_B=\l [\o^T]_B. 
\]
By a homothetic transformation, we can then assume that, as a basic cohomology class, 
\begin{equation}\label{normalization}
[\rho^T]_B=(2n+2)[\o^T]_B. 
\end{equation}
Hence there is a basic function $h$ on $M$ such that
\begin{equation}\label{e-2-11}
\rho^T+\i \p_B\bp_B h=(2n+2)\o^T,
\end{equation}
where $h$ is called  the \emph{(transverse) Ricci potential} of $\omega^T$, with the normalization
\[\int_M e^{-h}dv_g=1.\]
Note that $h$ can be considered as a basic function on $M$ or a function on  $X$ which is invariant under both $\xi$ and $r\p_r$. 
On the cone $X$,  the Ricci form $\rho_X$ satisfies
\[
\rho_X+\i \p\bp h=\rho^T-(2n+2)\o^T+\i \p_B\bp_B h=0.
\]
If the Sasaki metric is Einstein, one easily sees that \eqref{normalization} holds and $h=0$.\\

Next we consider more general deformations by allowing the Reeb vector field to vary in a fixed abelian Lie algebra. First we recall \emph{Type-I deformation} defined in \cite{BGM}. Let $(M, \xi_0, \eta_0, g_0)$ be a compact Sasaki manifold, denote its automorphism group by  $\text{Aut}(M, \xi_0, \eta_0, g_0)$.  We remark that here $\text{Aut}(M, \xi_0, \eta_0, g_0)$ denotes the group of all {\em $C^1$} diffeomorphisms of $M$ that preserves $\xi_0, \eta_0, g_0$. Elements in $\text{Aut}(\xi_0, \eta_0, g_0)$ are automatically $C^{k+1}$ if $g_0$ is in $C^k$ ($k\geq 1$) \cite{CH}; in particular, they are automatically smooth if $g_0$ is smooth. Then  $\text{Aut}(M, \xi_0, \eta_0, g_0)$ acts on $(X, J)$ naturally and it is a subgroup of $\Aut(X, J)$. 
We fix a maximal torus $\T\subset \Aut(M, \xi_0, \eta_0, g_0)$. 

\begin{defi}[Type-I deformation]
Let $(M, \xi_0, \eta_0, g_0)$ be a $\T$-invariant Sasaki structure and let $\T\subset \Aut(M, \xi_0,  \eta_0, g_0)$ be a compact maximal torus. For any $\xi\in \ft$ such that $\eta_0(\xi)>0$. We define a 
 new Sasaki structure on $M$ explicitly as 
\begin{equation}\label{e-2-7}
\eta=\frac{\eta_0}{\eta_0(\xi)}, \Phi=\Phi_0-\Phi_0\xi\otimes \eta, g=\eta\otimes \eta+\frac{1}{2}d\eta(\I\otimes \Phi).
\end{equation}
\end{defi}

It is clear from \eqref{e-2-7} that the family of Sasaki structures depend smoothly ($C^k$) on the Reeb vector field if the Sasaki structure $(\xi_0, \eta_0, g_0)$ is smooth ($C^k$).
It was proved directly in \cite{BGM} that the deformation given in \eqref{e-2-7} preserves the CR structure $(\cD=\Ker(\eta_0), \Phi_0|_{\cD})$.

It turns out that there is an equivalent description of Type-I deformation on the cone $(X, J)$.  In particular, the underlying complex cones corresponding to Type-I deformation are the same. As in \cite{MSY}, we fix a compact torus $\T\subset \Aut(X, J)$ and vary the Reeb vector fields within the Lie algebra $\ft$ of $\T$.  We fix a $\T$ invariant Sasaki metric $(M, \xi_0, \eta_0, g_0)$ with $\xi_0\in \ft$. We have the following,
 \begin{lem}\label{l-2-6}
 For any $\xi\in \ft$ such that $\eta_0(\xi)>0$, then there exists a $\T$ invariant K\"ahler cone metric $g_X$ on $(X, J)$ with Reeb vector field $\xi$. 
  \end{lem}

\begin{proof} 
Let $r_0^2$ be the K\"ahler potential on $(X, J)$ corresponding to the Sasaki structure $(M, \xi_0, \eta_0, g_0)$. To determine the required K\"ahler cone metric, we need to describe its radial function $r$.
We notice $\langle -J\xi, r_0^{-1}\p_{r_0}\rangle=\eta_0(\xi)$, which is positive and $r_0$ invariant, hence is uniformly positive and bounded. So for any $x\in X$, the integral curve $\phi_t(x)$ of the vector field $-J\xi$ on $X$ is always transverse to the link $r_0=constant$ at an angle strictly between $0$ and $\pi$. In particular $\lim_{t\rightarrow-\infty}r_0(\phi_t(p))=0$ and $\lim_{t\rightarrow\infty}r_0(\phi_t(p))=+\infty$.  
Now we identify $M$ with $\{r_0=1\}$ and thus $M\times\R^+$ with $X$ using the cone metric corresponding to $(\xi_0, \eta_0, g_0)$ (as in the beginning of this section).  Then we obtain a diffeomorphism $F: M\times\R^+\rightarrow M\times\R^+$ which sends $(p, r_0)$ to the point $\phi_{\log r_0}(p)=(q(p, r_0), r(p, r_0))$.   It is clear from the definition and the above identification of $X$ with $M\times \R^+$ that we may view $(q, r)$ as new coordinates on $X$, and we have $r\p_r=-J\xi$.   
We note that $r$ is $\T$ invariant and
 $\lim_{r_0\rightarrow 0(+\infty)} r=0(+\infty)$.  Define
\[
\o=\i \p\bp \left(\frac{r^2}{2}\right)\; \mbox{and}\;
g_X= \o(\cdot, J\cdot).
\]
Set $\eta=J(r^{-1}dr)$. We have
\begin{equation}\label{E2.6-2}
\omega=\frac{1}{2}d(r^2 \eta); g_X=dr^2+r^2 (\eta\otimes \eta+\frac{1}{2}d\eta(\I\otimes J))
\end{equation}

Note that the link $\{r=1\}$ coincides with $\{r_0=1\}$. On this link,  $V\in \Ker(\eta)$  if and only if $JV$ is tangent to $\{r=1\}$, which is the same as $JV$ is tangent to $\{r_0=1\}.$ So this is equivalent to  $V\in \Ker(\eta_0)$. It then follows that on $\{r_0=1\}$, there exists  a function $f$ so that $\eta=f \eta_0. $
Indeed,  $f=1/\eta_0(\xi)>0.$ We claim that $g_X$ is  positive definite and defines a K\"ahler cone metric on $X$. First it is readily seen that $\frac{\p}{\p r}$ and $r^{-1}\xi$ are orthogonal and of unit norm. Moreover, for any $Y\in \Ker (\eta)$, $g_X(r\p_r, Y)=g_X(\xi, Y)=0$. So it suffices to check $g_X$ is positive definite on $\Ker(\eta)$. Note that by \eqref{E2.6-2}, $g_X$ (and $\omega$) has homothetic degree two with respect to $r\p_r$. Thus we only need to prove $g_X$ is positive definite on $\Ker(\eta)$ when restricted on $\{r=1\}$. For any two vectors $Y$, $Z\in\Ker(\eta)$ on $\{r=1\}$, we have
\[
d\eta(Y, JZ)=\frac{1}{\eta_0(\xi)} d\eta_0(Y, JZ). 
\]
Note that we also have $\cL_\xi r=dr(\xi)= dr (J(r\p_r))=0$. It follows that $(M\times \R^{+}, J, g_X)$ defines a K\"ahler cone metric with radial function $r$ and Reeb sector field $\xi$. 

The construction above is indeed equivalent to Type-I deformation given in \eqref{e-2-7}. 
We restrict $\eta$ as a 1-form on the link $\{r_0=1\}$ and we have $\eta=\eta_0/\eta_0(\xi)$.
We can then directly compute that $d\eta(\xi, \cdot)=0$ using the fact that $\eta_0$ is $\xi$-invariant. Hence $\eta$ defines a contact 1-form on $\{r_0=1\}$ with the Reeb vector field $\xi$. 
Let $\Phi Y=J Y, Y\in\Ker(\eta)$ and $\Phi \xi=0$. Then $(\eta, \xi, \Phi)$ is a compatible triple on $\{r_0=1\}$ and it defines a Sasaki metric by 
\[
g=\eta\otimes \eta+\frac{1}{2}d\eta(\I\otimes \Phi).
\]
Now we construct a cone metric $g_X=dr^2+r^2 g$ on $(X, J)$ using the function $r$, and $g_X$ is the Kahler (cone) metric which corresponds to the Sasaki structure $(\eta, \xi, \Phi)$. 
 \end{proof}

It turns out $\eta_0(\xi)>0$ is also necessary for $\xi\in \ft$ being a Reeb vector field of a contact 1-form that comes out of deformation. 

 \begin{lem}\label{l-2-9}
Let $\eta(t) (t\in[0,1])$ be a continuous path of contact 1-forms with Reeb vector field $\xi(t)\in\ft$ such that $\eta(0)=\eta_0$ is the fixed $\T$-invariant contact 1-form, then $\eta_0(\xi(t))>0$ for all $t\in[0,1]$.
  \end{lem}

\begin{proof}The property that $\eta_0(\xi(t))>0$ is clearly an open property in $t$. So it suffices to prove this is also a closed property. Thus we can assume $\eta_0(\xi(t))>0$ for all $t\in[0,1)$, and we need to prove $\eta_0(\xi(1))>0$. For simplicity we denote $\xi=\xi(1)$. Suppose this not true, then by continuity, $\eta_0(\xi)\geq 0$ and there exists a point $p\in M$ such that $\eta_0(\xi)(p)=0$. Note that $\eta_0\wedge (d\eta_0)^n$ and $\eta\wedge (d\eta)^n$ are two volume forms on $M$. Then there exists a nowhere vanishing function $f$ such that
\[
f \eta_0\wedge (d\eta_0)^n= \eta\wedge (d\eta)^n. 
\] 
Clearly $f>0$. It follows that
\begin{equation}\label{e-2-8}
(d\eta)^n=\iota_\xi \eta\wedge (d\eta)^n=\iota_\xi (f \eta_0\wedge (d\eta_0)^n). 
\end{equation}
We compute
\[
\begin{split}
\iota_\xi (f \eta_0\wedge (d\eta_0)^n)=&f \eta_0(\xi) (d\eta_0)^n-f\eta_0 \wedge \iota_{\xi} (d\eta_0)^n\\
=&f \eta_0(\xi) (d\eta_0)^n-nf\eta_0\wedge \iota_\xi d\eta_0 \wedge (d\eta_0)^{n-1}.
\end{split}
\]
Note that $\eta_0$ is $\T$-invariant; in particular
\[
\cL_{\xi}\eta_0=\iota_{\xi}d\eta_0+d(\eta_0(\xi))=0.
\]
So
\begin{equation}\label{e-2-9}
\iota_\xi (f \eta_0\wedge (d\eta_0)^n)=f \eta_0(\xi) (d\eta_0)^n+nf\eta_0\wedge d(\eta_0(\xi)) \wedge (d\eta_0)^{n-1}.
\end{equation}
It then follows from \eqref{e-2-8} and \eqref{e-2-9} that
\begin{equation}\label{e-2-10}
(d\eta)^n=f \eta_0(\xi) (d\eta_0)^n+nf\eta_0\wedge d(\eta_0(\xi)) \wedge (d\eta_0)^{n-1}.
\end{equation}
Note that at $p$, $\eta_0(\xi)=0$ by assumption; also $d(\eta_0(\xi))(p)=0$ since $p$ is a minimum of $\eta_0(\xi)$. By \eqref{e-2-10}, it implies that $(d\eta)^n(p)=0$. Contradiction. 
\end{proof}

\begin{defi}
We define the associated  \emph{Reeb cone} $\cR_{\xi_0}$ of $\xi_0$ 
to be 
\[
\cR_{\xi_0}=\{\xi\in \ft: \eta_0(\xi)>0\}.
\]
\end{defi}

The notion of a \emph{Reeb cone} originates from \cite{MSY} Section 2.5, where it is defined to be the cone in $\ft$ dual to the moment cone $\mathcal C ^*$ in $\ft^*$. This is the path-connected component of $\xi_0$ in the set of all possible elements in $\ft$ that are Reeb vector field of some K\"ahler cone metric on  $(X, J)$.
By Lemma \ref{l-2-6} and Lemma \ref{l-2-9} the notion of a Reeb cone stated in Definition \ref{simple deformation} agrees with that given in \cite{MSY}. The notion of Reeb cone also coincides with the \emph{Sasaki cone} for a fixed CR structure introduced in \cite{BGS1}.  Note that a Reeb cone is always convex, and it is proved in \cite{MSY} that the element in $\cR_{\xi_0}$ that is the Reeb field of a Sasaki-Einstein metric is unique. 

\begin{defi}[Simple deformation]\label{simple deformation}
 We  define a \emph{simple deformation} of $(M, \xi_0, \eta_0, g_0)$ on $M$ with respect to $\T$ to be a Sasaki structure that is induced from a K\"ahler cone metric on $(X, J)$ with Reeb vector fields in $\cR_{\xi_0}$. Here we emphasize that when we talk about simple deformations, we fix the underlying complex structure on the K\"ahler cone, and we also need to specify a torus $\T$.
\end{defi}

By the above discussion we know a simple deformation is  the composition of a Type I deformation followed by a transverse K\"ahler deformation. For our purpose we shall also talk about a $C^{k, \alpha}$ simple deformation for $k\in \N$ and $\alpha\in (0, 1)$.
 \begin{defi}[$C^{k, a}$ simple deformation]\label{E4.2-d}
A $C^{k, \alpha}$ simple deformation is  a smooth type-I deformation (with respect to a fixed torus $\T$) followed by a transverse K\"ahler deformation given by a $C^{k+2, \alpha}$ deformation of the transverse K\"ahler potential. Notice that, by this definition, the Reeb vector fields involved all belong to the Lie algebra $\ft$, so are always smooth.  \end{defi}

\begin{rmk} It would be interesting to understand whether it is possible to have two different Reeb cones for fixed $(X, J)$ and $\ft. $
\end{rmk}

\subsection{Weighted Sasaki spheres}\label{S-2-5}

Consider the standard Sasaki structure on the sphere $S^{2n+1}$. The K\"ahler cone $X$ is $\C^{n+1}\setminus \{0\}$ with the flat metric $\omega=\frac{\sqrt{-1}}{2}\sum_{i=0}^n dz^i\wedge d\bar{z}^{i}$. The contact form on the link $M$ is given by  $\eta=\frac{1}{r^2}\sum_{i=0}^n (y^idx^i-x^idy^i)$, and the Reeb vector field is $\xi=\sum_{i=0}^n(y^i\frac{\p}{\p x^i}-x^i\frac{\p}{\p y^i})$.  Here $z^i=x^i+\sqrt{-1}y^i$. The  automorphism group of this Sasaki metric is $U(n+1)$. We take a maximal torus $\T^{n+1}$ in $U(n+1)$ consisting of diagonal elements. The Lie algebra $\ft$ is generated by the elements $\xi_i=y^i\frac{\p}{\p x^i}-x^i\frac{\p}{\p y^i}$. For any $\xi'=\sum_{i=0}^n a_i\xi_i$, we have 
$\eta(\xi')=\sum_{i=0}^na_i|z^i|^2$. So for  $\xi'$ to be positive,  it is equivalent that $a_i>0$ for all $i$. Thus the Reeb cone in this case  is  $\R_{+}^{n+1}$.

We call a Sasaki structure on $S^{2n+1}$ \emph{simple} if it is isomorphic to a Sasaki structure on $S^{2n+1}$ that comes out of a simple deformation from the standard Sasaki structure $(\xi, \eta, g)$. The K\"ahler cone of a simple Sasaki structure is bi-holomorphic to $\C^{n+1}\setminus\{0\}$ and the corresponding Sasaki manifold is called a \emph{weighted Sasaki sphere}.   All simple Sasaki structures form a connected family of Sasaki structures on $S^{2n+1}$. 
For $a=(a_0, \cdots, a_n)\in \R^{n+1}_{+}$, we denote by $\xi_a=\sum_i a_i\xi_i$. 
It is not hard to see that a simple Sasaki structure on $S^{2n+1}$ with Reeb vector field $\xi_a$ is quasi-regular if and only if $a \in \Q_{+}^n$, in which case we get a circle bundle over a weighted projective space. It is regular precisely when all the $a_i$'s are equal, in which case the Sasaki structure is isomorphic to the standard one on $S^{2n+1}$ up to a homothetic transformation.

\subsection{Sasaki-Ricci solitons} \label{S-2-3}
In this subsection we recall some general theory on the symmetries in Sasaki geometry, largely following \cite{MSY} and \cite{FOW}. We also state some facts about Sasaki-Ricci solitons, whose proofs follow analogously  the K\"ahler case as in \cite{Tian-Zhu2, Tian-Zhu}. Let $(M, \xi, \eta, g)$ be a Sasaki manifold, and $(X, g_X, J)$ the corresponding K\"ahler cone.  Denote $G=\text{Aut}(\xi, \eta, g)$,  and denote by $\g$ the Lie algebra of $G$. 

\begin{defi} 
We say a vector field $Y$ on a Sasaki manifold $(M, \xi, \eta, g)$ is a \emph{Hamiltonian holomorphic vector field} if its homogeneous extension to $X$ (which we also denote by $Y$) is a Hamiltonian holomorphic vector field with respect to the K\"ahler metric $(J, g_X)$.  
\end{defi}

This definition is essentially the same as the one given in \cite{FOW}, where it is phrased in terms of transverse K\"ahler geometry. The difference is that in \cite{FOW} the vector fields are allowed to be complex valued, while in this definition we only consider real valued vector fields. Hence a Hamiltonian holomorphic vector field is clearly a Killing field. For a Hamiltonian holomorphic vector field $Y$, we have $\cL_Y r=0$ and 
$\cL_{r\p_r} Y=0, $
Then $$\iota_Y \omega=\frac{1}{2}\iota_Y d(r^2\eta)=\frac{1}{2}\cL_Y (r^2\eta)-\frac{1}{2} d(r^2\eta(Y))=-\frac{1}{2}d(r^2\eta(Y)).$$
So $Y$ is generated by the Hamiltonian function $H_Y=-\frac{1}{2} r^2\eta(Y)$, i.e. 
$$Y=\frac{1}{2}J\nabla_X (r^2\eta(Y))=\eta(Y)\xi+\frac{r^2}{2}J\nabla_X\eta(Y). $$
We call the function $H_Y$ a \emph{Hamiltonian holomorphic potential} on $M$. 
It is easy to see that  the Lie algebra $\g$ can be identified with the Lie algebra of Hamiltonian holomorphic vector fields $Y$ on $M$, or equivalently, the Lie algebra of Hamiltonian holomorphic potentials on $M$ under the Poisson bracket.

Now we fix the homothetic re-scaling by equation (\ref{normalization}). Note this normalization is what one should use if one is searching for Sasaki-Einstein metrics. Let $h$ be the Ricci potential of $\o^T$, as defined in (\ref{e-2-11}). There is an alternative characterization of a Hamiltonian holomorphic vector field in terms of a self-adjoint operator on $M$ with respect to the measure $e^{-h}dv$. The operator is given by 
$$L(\psi)=\t \psi-\nabla h\cdot \nabla \psi+4(n+1)\psi, $$ for a basic function $\psi$, where $\t$ is the rough Laplacian of $g$. Note that for a basic function $\psi$, $\t \psi=-\t_B \psi$.   The corresponding operator on the cone $X$ is given by 
$$L_X(\psi)=r^2(\t_X\psi-\nabla_X h\cdot \nabla_X\psi)+4(n+1)\psi.$$

\begin{defi} We call a basic function $\psi$ on $M$ \emph{normalized} if 
\begin{equation}\label{normalized potential}
\int_M \psi e^{-h}dv_g=0
\end{equation}
\end{defi}
For an arbitrary basic function $\psi$, we denote by 
\begin{equation}\label{T-4-n}
\underline{\psi}=\psi-\int_M \psi e^{-h}dv_g
\end{equation} 
the normalization of $\psi$.
A straightforward Bochner technique gives 
 
\begin{lem}[\cite{FOW}]\label{l-2-4-2}
A real-valued basic function $\psi$
 satisfies $L(\underline{\psi})=0$ if and only if $\psi$ is a Hamiltonian holomorphic potential. 
\end{lem}
 Later we will also consider \emph {complex Hamiltonian holomorphic potentials}, which are complex-valued basic functions  $\psi$ which satisfies that $L(\underline{\psi})=0$. Similar as before, the space of all complex Hamiltonian holomorphic potentials can be identified with the Lie algebra of the group $P$ of holomorphic transformations of $(X, J)$ that commute with the dilation generated by $r\p_r$. In literature, $P$ is often called \emph{the transverse holomorphic automorphism group}.  Clearly $P$ does not change under the transverse K\"ahler deformation of the Sasaki structure. 
 
For any $Y\in \g$, we define the \emph{Futaki invariant} of $JY$ as 

$$Fut(JY)=\int_{X} \cL_{JY}(-h)e^{-\frac{r^2}{2}}dV_X=\frac{1}{2}\int_X \nabla_X h\cdot \nabla_X \eta(Y) r^2 e^{-\frac{r^2}{2}}dV_X.$$
The appearance of the exponential term guarantees the validity of integration by parts, and one can show that this does not change under a transversal K\"ahler deformation of the metric. 

Now we introduce the notion of a \emph{Sasaki-Ricci soliton}, following \cite{FOW}. 
\begin{defi} A compact Sasaki manifold $(M, \xi, \eta, g)$ is called a Sasaki-Ricci soliton if there is a Hamiltonian holomorphic vector field $Y$ such that 
$$Ric^T-(2n+2)g^T=\cL_{JY}g^T. $$
\end{defi}
This is equivalent to that  the Ricci potential $h$ is a Hamiltonian holomorphic potential, in other words, $L(\underline{h})=0$.  
 Suppose $\T$ is a maximal torus of $G$ and $(M, \xi, \eta, g)$ is a Sasaki-Ricci soliton, then it is Sasaki-Einstein if $Fut(JY)=0$ for all $Y\in \ft$. Indeed,  the Hamiltonian vector field $Y_h$ generated by the Ricci potential $h$ commutes with all elements in $\ft$, and by maximality of $\T$,  it must be in $\ft$. So $Fut(JY_h)=\frac{1}{2}\int_X |\nabla_X h|^2e^{-\frac{r^2}{2}} dV_X=0$, and thus $h$ is constant. 

\subsection{The normalization} \label{S-2-4}
By definition a Sasaki-Ricci soliton satisfies the normalization condition (\ref{normalization}), and this fixes the homothetic transformation of the Reeb vector field.  Thus when we deform Sasaki-Ricci solitons, one would like to ask under what condition on $\xi$ is (\ref{normalization})  preserved. 

\begin{defi} 
We say a compact Sasaki manifold $(M, \xi, \eta, g)$ is \emph{normalized} if it satisfies (\ref{normalization}). i.e.
$$2\pi c_1^B=(2n+2)[\omega^T].$$
\end{defi}

By \cite{FOW} a Sasaki manifold is homothetic to a normalized one if and only if the basic first Chern class $c_1^B$ is positive definite and the contact subbundle $\cD$ has vanishing first Chern class. We have already seen in Section 2.2 that if $(M, \xi, \eta, g)$ has positive transverse bisectional curvature, then it can be normalized by a homothetic transformation. 

Now we fix a compact  normalized Sasaki manifold $(M, \xi_0, \eta_0, g_0)$,  thus a K\"ahler cone $(X, J, g_X)$, and also a maximal torus $\T$ in $\Aut(\xi_0, \eta_0, g_0)$ whose Lie algebra contains $\xi_0$. By the transverse Calabi-Yau theorem we may assume $Ric^T$ is positive, so by Myers' theorem the fundamental group of $M$ is finite, as a homothetic transformation would produce a metric on $M$ with positive Ricci curvature.  On the K\"ahler cone we know $Ric_X(\omega_0)+\sqrt{-1}\p\bp h_0=0$. So $e^{-h_0}\omega_0^n$ defines a flat connection on $K_X$.  Parallel transport then defines a nowhere vanishing holomorphic section $\Omega_l$ of $K^{\otimes l}$ for some integer $l>0$. Moreover, since $\cL_{\xi_0}h_0=\cL_{r_0\p_{r_0}}h_0=0$, we get 
$$\cL_{r_0\p _{r_0}}\Omega_l=(n+1)l\Omega_l;$$
in other words, using the language in \cite{MSY}, $\Omega_l$ has charge $(n+1)l$ for the vector field $r_0\p_{r_0}$. 
Fixing a choice of $\Omega_l$, then we claim there is a linear functional $c: \ft\rightarrow\R$ such that for any $Y\in\ft$ we have
$$\cL_{Y}\Omega_l=\sqrt{-1}c(Y)\Omega_l.$$
To see this, for any $Y$ there is a holomorphic function $f$ on $X$ such that 
$$\cL_Y\Omega_l=f_Y\Omega_l.$$
Since $Y$ commutes with $r_0\p_{r_0}=-J\xi_0$, we know that $\cL_{r_0\p_{r_0}} f_Y=0$. But the only holomorphic function on $X$ with homothetic degree zero is the constant function, so
$f_Y$ is  constant. It is purely imaginary because the $\T$ action preserves the norm of $\Omega_l$. Now we have

\begin{lem} A $\T$ invariant K\"ahler cone metric on $(X, J)$ with Reeb vector field $\xi\in \ft$ is normalized if and only if $c(\xi)=(n+1)l$. 
\end{lem}
\begin{proof} Let $\omega$ be a such a metric. Then the Ricci curvature is given by 
$$Ric_X(\omega)=-\sqrt{-1}\p\bp \log {\omega}^{n+1}=\sqrt{-1}\p\bp \log ||\Omega_l||_{\omega}^{2/l}.$$
If $c(\xi)=(n+1)l$, then $h=-\log ||\Omega_l||_{\omega}^{2/l}$ is both $\T$-invariant and $J\xi$ invariant, so  the transverse Ricci curvature satisfies
$$Ric^T(g)+\sqrt{-1}\p_B\bp_B h=(2n+2)\omega^T,$$
and equation (\ref{normalization}) thus holds. Since on a fixed ray $\R_{+}\cdot \xi$ there is at most one possible Reeb vector field satisfying (\ref{normalization}), we conclude the lemma. 
\end{proof}

 From now on we define $\mathcal H$ to be the hyperplane of vectors $\xi$ in $\ft$ satisfying $c(\xi)=(n+1)l$. Then one has an explicit characterization of the tangent space of $\mathcal H$ at $\xi_0$.

\begin{lem} An element $Y\in\ft$ satisfies $c(Y)=0$ if and only if 
\begin{equation}\label{e-2-12}
\int_M \eta_0(Y) e^{-h_0}dv_{g_0}=0.
\end{equation}
\end{lem}
\begin{proof} Any such $Y$ is of the form $$Y=\frac{1}{2}J\nabla_X(r_0^2\eta_0(Y)).$$
So $$\cL_{JY}h_0=-\cL_{JY}\log ||\Omega_l||_{\omega_0}^{2/l}=-\frac{1}{2}\Delta_X(r_0^2\eta_0(Y))+\frac{2}{l}c(Y).$$ That is 
$$\frac{1}{2}\Delta_X(r_0^2\eta_0(Y))=\frac{1}{2}\nabla_X h_0\cdot \nabla_X(r_0^2\eta_0(Y))+\frac{2}{l}c(Y).$$
Lemma \ref{l-2-4-2} implies that $c(Y)=0$ if and only if $\int_M \eta_0(Y)e^{-h_0}dv_{g_0}=0.$
\end{proof}

\section{Proof of the main theorem} \label{S-3}

\subsection{Volume functional} \label{S-3-2}
As in Section \ref{S-2-4} we fix  a normalized  Sasaki manifold $(M, \xi_0, \eta_0, g_0)$ and a maximal torus $\T$ in $\Aut(\xi_0, \eta_0, g_0)$, so this defines the Reeb cone $\cR$ and the hyperplane $\cH$. 
 Now recall from \cite{MSY} the volume of a Sasaki manifold is invariant under transverse K\"ahler deformation, so it gives rise to a functional 
 $$\Vol:\cR\rightarrow \R_+. $$
For $Y\in \ft$, we have the first variation formula
 $$\delta_Y\Vol|_{\xi_0}=-2\int_{X} r_0^2\eta_0(Y)e^{-\frac{r_0^2}{2}}dV_X.$$
If $Y$ is normalized, i.e.  $Y$ satisfies \eqref{e-2-12}, 
then we have
 \begin{eqnarray*}
 Fut_{\xi_0}(JY)&=&\frac{1}{2}\int_X[ 4(n+1)\eta_0(Y)+r_0^2\Delta_X\eta_0(Y)]e^{-\frac{r_0^2}{2}}dV_X\\
 &=&2(n+1)\int_X\eta_0(Y)e^{-\frac{r_0^2}{2}}dV_X.
 \end{eqnarray*}
 So we obtain
 \begin{prop}[\cite{MSY}] \label{volume gradient}
 The gradient of the volume functional is given by the Futaki invariant. More precisely, for any normalized $Y\in \ft$, 
  \begin{equation}\label{e-2-13}
 \delta_Y \Vol |_{\xi_0}=-2Fut_{\xi_0}(JY).
 \end{equation}
 \end{prop}
This implies, in particular, the Reeb vector field of a Sasaki-Einstein metric is a critical point of the volume functional restricted on $\mathcal H$.

Recall the Reeb cone $\cR$ is an open convex subset of $\ft$. With respect to a fixed Euclidean metric on $\ft$, we have 
\begin{prop}[\cite{MSY}] \label{convexity}
The Hessian of $\Vol$ is strictly positive definite on $\cR$. 
\end{prop}

 Now we denote by $\cR'=\cR\cap \cH$ the set of all normalized Reeb vector fields. Then the proposition in particular implies that  the volume functional is a strictly convex function on the convex set $\cR'$, so its critical point, if exists, must be unique. In the toric case, it is proved in \cite{MSY1}, \cite{FOW} that the volume functional is proper, so the critical point indeed exists.  In general, we have

\begin{prop}\label{P-p}
The volume functional is proper on $\cR'$. In particular it always has a unique minimizer. 
\end{prop}

\begin{proof} We first claim that $\cR'$ is bounded in $\cH$.  Fix an arbitrary norm $||\cdot||$ on $\ft$. Since $\cR$ is open, there is a $\delta>0$ such that any $\xi\in \cR'$ with $||\xi-\xi_0||\leq \delta$ lies  in $\cR'$.
Recall by Lemma \ref{e-2-12}  for any $\xi\in \cR'$, 
\[
\int_M (\eta_0(\xi-\xi_0))e^{-h_0} \eta_0\wedge (d\eta_0)^n=0.
\]
 So there is an $\epsilon>0$ such that   for any $\xi\in \cR'$ with $||\xi-\xi_0||=\delta$ we have 
$$\min \eta_0(\xi-\xi_0)\leq -\epsilon<0.$$
For any $\xi\in \cR'$, we denote 
$$\xi'=\xi_0+\frac{1}{||\xi-\xi_0||}(\xi-\xi_0).$$ 
Then $\xi'\in \cR'$, and 
$$\min \eta_0(\xi'-\xi_0)=\frac{1}{||\xi-\xi_0||} (\min \eta_0(\xi)-1)\geq -\frac{1}{||\xi-\xi_0||}.$$
So 
$$||\xi-\xi_0||\leq 1/\epsilon.$$
This proves the claim. Now to prove the proposition, it suffices to show that for any sequence $\{\xi_i\}$ in $\cR'$ with $\lim_{i\rightarrow\infty} \xi_i=\xi\in \partial \cR'$ we have
$$\lim_{i\rightarrow\infty}Vol(\xi_i)=+\infty.$$
Clearly 
$\eta_0(\xi)\geq 0$ and $\min \eta_0(\xi)=0$. From the discussion in Section \ref{S-2-2} we know that  there is the Sasaki structure on $M$ given by the type I deformation with $(\xi_i, \eta_i=\frac{\eta_0}{\eta_0(\xi_i)})$. So
\[
\Vol(\xi_i)= \int_M \eta_i \wedge (d\eta_i)^n=\int_M (\eta_0(\xi_i))^{-n-1} \eta_0\wedge (d\eta_0)^n. 
\]
Since $\eta_0(\xi_i)$ is positive for any $i$, by Fatou's lemma, we have
\[
\int_M (\eta_0(\xi))^{-n-1} \eta_0 \wedge (d\eta_0)^n \leq \liminf_{i\rightarrow \infty} V(\xi_i). 
\] 
So we only need to show the integral in the left hand side is unbounded. Suppose $\eta_0(\xi)(p)=0$ at some point $p\in M$. Then $d(\eta_0(\xi))(p)=0$.  We choose a local coordinate chart $(x, z_1, \cdots, z_n)$ around $p$, such that $\xi_0=\p_{x}$, and $(z_1, \cdots, z_n)$ is a local transverse holomorphic coordinate. Since $\xi\in \ft$, $\eta_0(\xi)$ is a basic function with respect to $\xi_0$. Thus for $\epsilon, r>0$ sufficiently small, on $U_{\epsilon, r}=\{((x, z_1, \cdots, z_n)| |x|\leq \epsilon, \sum|z_i|^2\leq r^2 \}$ we have
$$0\leq \eta_0(\xi) \leq C \sum_i |z_i|^2. $$ Hence
\[
\int_M (\eta_0(\xi))^{-n-1} \eta_0\wedge (d\eta_0)^n \geq C \int_{|x|\leq \epsilon} \int_{\sum_i |z_i|^2\leq r^2}(\sum_i|z_i|^2)^{-n-1}dzd\bar{z}=+\infty.
\]\end{proof}

\subsection{Perelman's $\mu$ entropy} \label{S-3-3}

Perelman's entropy functionals are the key ingredients in studying Ricci flow and Ricci solitons. 
We  briefly recall these functionals in the Sasaki setting, which were introduced in  \cite{Collins, He} to study the Sasaki-Ricci flow.  
Let $(M, \xi, \eta, g)$ be a compact Sasaki manifold. The $W$ functional is defined as usual, 
\begin{equation}\label{E-4-1}
W(g, f)=\int_M e^{-f}(R+|\nabla f|^2+4(n+1)f)dv_g,
\end{equation} 
and we define the $\mu$ functional by
\[
\mu(g)=\inf \left\{ \cW (g, f): df(\xi)=0,  \int_M e^{-f} dv_g=1.\right\}
\]
The difference from the standard $\mu$ functional is that first we only consider basic functions, so $df(\xi)=0$, and second we have assigned the scaling constant $\tau$ to be $\frac{1}{4(n+1)}$.
One can show that in the definition of $\mu$ there always exists a smooth minimizer $f$ (see \cite{He} Section 9 for example). As usual, $g$ is a critical point of $\mu$ on the space of all transverse K\"ahler deformations if and only if it is a Sasaki-Ricci soliton (It is here that the particular of choice $\tau$ enters).  Furthermore, if $(M, \xi, \eta, g)$ is a Sasaki-Ricci soliton,  by a straightforward extension of Corollary 1.5 in \cite{TZZZ}, fixing a maximal torus $\T$ in $\Aut(\xi, \eta, g)$ (which must contain the one parameter subgroup generated by the potential Killing vector field $Y_h$),  $g$ is the maximum of $\mu$ among all $\T$ invariant transverse K\"ahler deformations. 

As before, we fix  a normalized  Sasaki manifold $(M, \xi_0, \eta_0, g_0)$ and $\T$. We define a functional $I: \cR'\rightarrow \R$ by
  \[I(\xi)=\max_{(\xi, \eta, g)} \mu(g)\] among all $\T$-invariant transverse K\"ahler deformations $(\xi, \eta, g)$. By Lemma \ref{l-2-6}, for any $\xi\in \cR^{'}$ we can choose a $\T$-invariant type I deformation of $(\xi_0, \eta_0, g_0)$ induced by $\eta_0(\xi)$ and we define $\mu_0(\xi)$ to be the $\mu$-functional of such a Sasaki structure. By definition we have $I(\xi)\geq \mu_0(\xi)$. So we have, combining Corollary 1.5 in \cite{TZZZ} mentioned above,  

\begin{prop} \label{mu uniform bound}
For any compact subset $K$ in $\cR'$, there is a number $c_K$ so that $I(\xi)\geq c_K$ for all $\xi\in K$. In particular, if $(\xi, \eta, g)$ is a $\T$ invariant Sasaki-Ricci soliton with Reeb vector field $\xi\in K$, then $\mu(g)\geq c_K$. 
\end{prop}

Our original strategy to prove Theorem \ref{T-main} was to run the gradient flow of the volume functional for Reeb vector fields in $\cR'$ and to deform a Sasaki-Ricci soliton correspondingly, so it is natural to consider how the $\mu$ functional behave under such deformations.  Although no longer necessary for the purpose of this paper, we include a brief discussion here.  
As before we fix a maximal torus $\T$, with Lie algebra $\ft$, and we only consider simple deformations with respect to $\T$.  Let $f$ be a minimizer function of $W(g, \cdot)$. Then $f$ is smooth (c.f. \cite{He}).  
Suppose the infinitesimal deformation 
$\delta\xi=Y\in \ft$ is generated by a normalized Hamiltonian potential $u$,
then 
\begin{prop} \label{T-4-2} Denote 
$
c_n=(2n+2)2^{n-1}(n-1)!,
$
we have
\begin{equation}\label{E-4-28}
\begin{split}
c_n^{-1}\delta \mu(g)=&2\int_X e^{-f-\frac{r^2}{2}}\left((\nabla_X f, \nabla_X u)-u\right)dV_X\\
&-\int_X e^{-f-\frac{r^2}{2}} r^2(\delta g^T, Ric^T+Hess^T f) dV_X.
\end{split}
\end{equation}
\end{prop}
Here $f$ and $u$ are trivially extended to functions over $X$, $dV_X$ is the volume form on $X$ given by by $(2n+2)r^{2n+1}dr\wedge dv_g$,  the inner product and gradient are taken with respect to the K\"ahler cone metric $g_X=dr^2+r^2g$, and $\delta g^T$ is the transverse projection of $\delta g$.  A convenient way to prove Proposition \ref{T-4-2} is to express the terms on the K\"ahler cone and compute the variation on the cone. 
The computation is straightforward but rather involved. Since we do not need this result, we refer to our previous preprint \cite{HS1}  for the details of computation.   Proposition \ref{T-4-2} has some interesting consequences. For example we have the following variational characterization of Sasaki-Einstein metrics, 

\begin{cor}
A compact Sasaki manifold $(M, g)$ is Einstein if and only if it is a critical point of $\mu$ among all simple deformations with respect to $\T$.\end{cor}
\begin{proof}
It directly follows from  \eqref{E-4-28}, when restricted to transverse Kahler deformations,  that if $g$ is a critical point, then it has to be a Sasaki-Ricci soliton, so that  the Ricci potential $h$ is the minimizer function. Then  let $u=\underline h$ (see (\ref{T-4-n}) and $Y$ be the Hamiltonian vector field generated by $u$. Then by maximality of $\T$, we know $Y\in\ft$.  Plug this $u$ into formula \eqref{E-4-28}, we obtain $\int_X e^{-h-\frac{r^2}{2}} |\nabla_X u|^2dV_X=0$. So $u$ is constant and so  is $h$. 
\end{proof}

\subsection{The main argument} \label{S-3-4}
Now we present the outline of our strategy in the proof of Theorem \ref{T-main}, leaving technical details to Section 4 and Section 5. First we classify compact Sasaki-Ricci solitons with positive transverse bisectional curvature. We have,

\begin{thm}\label{positive soliton}
A simply-connected compact Sasaki-Ricci soliton with positive transverse bisectional curvature is a simple Sasaki structure on $S^{2n+1}$. 
\end{thm}

\begin{proof}
Let $(M, \xi, \eta, g)$ be such a  Sasaki-Ricci soliton. As before we choose a maximal torus $\T$ of $\Aut(\xi, \eta, g)$ such that its Lie algebra $\ft$ contains $\xi$.  Let $\cR'=\cR\cap \cH$ be the space of normalized elements in the Reeb cone $\cR$.   By Proposition \ref{P-p},  the volume functional $\Vol$ has a unique minimizer $\xi_0\in \cR'$. Let $V_0=\Vol(\xi_0)$ and $V_1=\Vol(\xi)$. Denote by $\mathcal A$ the set of all real numbers $V\in [V_0, V_1]$ such that there is a Sasaki-Ricci soliton with positive transverse bisectional curvature that is a simple deformation of $(\xi, \eta, g)$ with respect to  $\T$, with volume $V$, and with maximal torus-symmetry $\T$. 
To prove Theorem \ref{positive soliton}, it suffices to show $V_0\in \mathcal A$. If there is such a Sasaki-Ricci soliton with volume $V_0$, then by Proposition \ref{convexity} and Proposition \ref{volume gradient},  its Reeb vector field must be 
$\xi_0$ and it has vanishing Futaki invariant. Hence it is indeed a Sasaki-Einstein metric.   Lemma \ref{SE positive} then implies  it must be isomorphic to the standard Sasaki structure on $S^{2n+1}$. 

In order to show $V_0\in \mathcal A$, we want to prove that $\mathcal A$ is both open and closed in $[V_0, V_1]$. The openness follows from Theorem \ref{local deformation},  where a local deformation theorem for Sasaki-Ricci solitons is proved.  To prove the closedness, we assume there is a sequence $V_i\in \mathcal A$, with $V_i\rightarrow V_\infty$. Let $(\xi_i, \eta_i, g_i)$ be the corresponding Sasaki-Ricci soliton with positive transverse bisectional curvature. By the volume properness (Proposition \ref{P-p}), the set of $\xi_i's$ is relatively compact in $\cR'$, so we may apply Proposition \ref{mu uniform bound} to conclude that there is a uniform lower bound on $\mu(g_i)$.   Then by Theorem \ref{T-5-2} this sequence has uniformly bounded curvature (and their covariant derivatives), volume and diameter. By Cheeger-Gromov convergence theory for Riemannian manifolds, by passing to a subsequence, there are diffeomorphisms $f_i$ so that $(\xi_i', \eta_i', g_i'):=f_i^*(\xi_i, \eta_i, g_i)$ converges to a limit Sasaki-Ricci soliton $(\xi'_\infty, \eta'_\infty, g'_\infty)$ in $C^\infty$ sense.
Clearly the limit has volume $V_\infty$. By Proposition \ref{P-8-1} this limit  has positive transverse bisectional curvature. 
Now we claim $V_\infty\in \cA$. However, $(\xi'_\infty, \eta'_\infty, g'_\infty)$ might  not be the desired Sasaki-Ricci soliton (with volume $V_\infty$).
 The problem is that, due to the gauge transformations $f_i$, the complex structure  on the K\"ahler cone corresponding to $(\xi'_\infty, \eta'_\infty, g'_\infty)$ might not be isomorphic to the original one  \emph{a prior}.

Now we show that after applying a suitable diffeomorphism, the complex structure of the Kahler cone corresponding to  $(\xi'_\infty, \eta'_\infty, g'_\infty)$ indeed coincides with the original one  and the maximal torus is given by $\T$. For this purpose,  we need to use two theorems proved in Section 4.

Fix a large $k$ and $\alpha\in (0, 1)$, and fix as usual a maximal torus $\T'$ in $\Aut(\xi_\infty', \eta_\infty', g_\infty')$. First we apply the rigidity Theorem \ref{rigidity},  which says that composing by a further sequence of  $C^{k+1, \alpha}$ diffeomorphisms if necessary (we still denote the resulting diffeomorphism by $f_i$),  we may assume for $i$ sufficiently large,  $(\xi'_i, \eta_i', g_i')$ is a $C^{k,\alpha}$ simple deformation of $(\xi_\infty', \eta_\infty', g_\infty')$ with respect to $\T'$. Note in particular $\xi_i'$ is in the Lie algebra of $\T'$, and converges to $\xi_\infty'$ smoothly. In particular, pulling back by ${f_i^{-1}}^*(\xi'_\infty, \eta'_\infty, g'_\infty)$ is  a simple deformation of $(\xi_i, \eta_i, g_i)$, but with respect to the torus ${f_i^{-1}}^*\T'$.  

To complete the argument we also need to relate this torus ${f_i^{-1}}^*\T'$ with $\T$. We claim that these two tori are conjugate to each other, for sufficiently large $i$. 
We apply the local deformation Theorem \ref{local deformation} to $(\xi'_\infty, \eta'_\infty, g'_\infty)$,  which asserts that for  sufficiently large $i$, we can obtain a Sasaki-Ricci soliton $(\xi_i', \eta_i'', g_i'')$ that is a simple deformation of $(\xi_\infty', \eta_\infty', g_\infty')$ with respect to $\T'$, and with maximal torus-symmetry $\T'$.  Choose a sufficiently big $i$ and fix it. Notice that  $(\xi_i', \eta_i', g_i')$ and $(\xi_i', \eta_i'', g_i'')$ have the same Reeb vector field and have the same the complex structure on their K\"ahler cones, hence they differ only by a transverse K\"ahler deformation. In particular
 they share the same transverse holomorphic automorphism group $P$. 
  Since they are both Sasaki-Ricci solitons,  by a Calabi type theorem \cite{Ca} for Sasaki-Ricci solitons (similar to the K\"ahler case proved in \cite{Tian-Zhu2}), both $\Aut(\xi_i', \eta_i', g_i')$ and $\Aut(\xi_i', \eta_i'',  g_i'')$ are maximal connected compact subgroups of $P$, and thus  these two groups, as well as their maximal tori, are conjugate in $P$. So there is an element $h\in P$ so that $Ad_h \T'=f_i^*\T. $ Now let $f=h\circ f_i^{-1}$ and $(\xi_\infty, \eta_\infty, g_\infty)=f^*(\xi_\infty', \eta_\infty', g_\infty')$. Then $(\xi_\infty, \eta_\infty, g_\infty)$ is  a simple deformation of $(\xi_i, \eta_i, g_i)$ with respect to $\T$, and has maximal torus-symmetry $\T$. This is the desired limit Sasaki-Ricci soliton with volume $V_\infty$. Hence $V_\infty\in \mathcal A$ and this finishes the proof.
\end{proof}

We can now prove Theorem \ref{T-main}.
\begin{proof}[Proof of Theorem \ref{T-main}]Let $(M, \xi, \eta, g)$ be a simply connected compact Sasaki manifold with positive transverse bisectional curvature satisfying the normalization condition  \ref{normalization}. We first run the Sasaki-Ricci flow $(\xi(t), \eta(t), g(t))$ from $(\xi, \eta, g)$.  As is proved in \cite{He}, the flow exists for all time and converges by sequence to a Sasaki-Ricci soliton in the sense of Cheeger-Gromov. By Proposition \ref{P-2-1}, we know it has positive transverse bisectional curvature.  Theorem \ref{positive soliton} then implies that any limit is a simple Sasaki structure on $S^{2n+1}$. By Theorem \ref{rigidity} again for $t$ large enough $(\xi(t), \eta(t), g(t))$ is also a simple Sasaki structure on $S^{2n+1}$. In particular, the corresponding K\"ahler cone is bi-holomorphic to $\C^{n+1}\setminus\{0\}$. Since the Sasaki-Ricci flow is a transverse K\"ahler deformation, we conclude that the original Sasaki structure $(\xi, \eta, g)$ is also a simple Sasaki structure on $S^{2n+1}$. \end{proof}

\begin{proof}[Proof of Corollary \ref{Frankel conjecture} and \ref{orbifold Frankel}] Given an $n$ dimensional compact K\"ahler manifold $(Z, J, \omega)$ with positive bisectional curvature, we denote by $l$ the Fano index of $Z$, and $\pi:M\rightarrow Z$ the unit circle bundle over $Z$ with first Chern class given by $c_1(Z)/l$. Then $M$ is simply connected. Let $\eta$ be a connection one form on $M$, i.e. $d\eta=\pi^*\omega$. Then $M$ endowed with the metric $g_M=\pi^*g+\eta\otimes \eta$ is a compact Sasaki manifold with positive transverse bisectional curvature.  Then by Theorem \ref{T-main} $(M, g_M)$ is a simple Sasaki structure on the sphere $S^{2n+1}$. But  there is only one possible Reeb vector field for a simple Sasaki structure on $S^{2n+1}$ to be regular, i.e. when $\xi$ is proportional to $(1, \cdots, 1)$. It is easy to see then $Z$ is bi-holomorphic to $\C\mathbb P^n$. The proof of Corollary \ref{orbifold Frankel} is similar, noting that a quasi-regular Sasaki manifold is precisely a $U(1)$ bundle over a K\"ahler orbifold  such that all the local uniformizing groups inject into $U(1)$(See \cite{BG} for example).\end{proof}

\begin{proof}[Proof of Theorem \ref{T-main-2}]   Using toric geometry, it is proved in \cite{FOW} that any simple Sasaki structure on $S^{2n+1}$ can be deformed to a  Sasaki-Ricci soliton through a transverse K\"ahler deformation.  We claim they all have positive transverse bisectional curvature. Indeed, since by the construction in \cite{FOW} the moduli space of simple Sasaki-Ricci solitons on $S^{2n+1}$ is connected and the standard Sasaki structure $(M, \xi, \eta, g)$ has positive transverse bisectional curvature, it suffices to prove that  the positivity condition is both open and closed. The openness is obvious, while the closedness follows  Proposition \ref{P-8-1}. 
\end{proof}

\section{Local deformation and rigidity}\label{S-4}
\subsection{local deformation of Sasaki-Ricci solitons} \label{S-4-1}
In this subsection we study the deformation theory of Sasaki-Ricci solitons under the variation of Reeb vector fields.  Fix a compact Sasaki-Ricci soliton $(M, \xi_0, \eta_0, g_0)$ and its K\"ahler cone $(X, J, g_X)$.  Let $\Aut(\xi_0, \eta_0, g_0)$ be its automorphism group, and $\g$ be its Lie algebra.  Let $\ft$ be a maximal abelian Lie sub-algebra of $\g$ that contains $\xi_0$, and $\T$ the maximal torus of $\Aut(\xi_0, \eta_0, g_0)$ that is generated by $\ft$. 
By assumption $$Ric_0^T+\sqrt{-1}\p_{B,0}\bp_{B,0} h_0=(2n+2)\omega_0^T $$ 
with $L_{\omega_0}(\underline{h_0})=0$,  where $L_{\omega_0}$ is the modified Laplacian operator defined with respect to the metric $\omega_0$, as in Section \ref{S-2-3}. Let $Y_0\in\g$ be the Hamiltonian holomorphic vector field generated by $h_0$.  Since $g_0$ is $\T$ invariant, we know $Y_0$ commutes with all elements in  $\ft$. By maximality we see $Y_0$ lies in $\ft$. As in Section \ref{S-2-4}, we denote by $\mathcal H$ the hyperplane of normalized elements in $\ft$.  Then the deformation of Sasaki-Ricci solitons is unobstructed:

\begin{thm}\label{local deformation} There is a neighborhood $\cU$ of $\xi_0$ in $\mathcal H$ and a smooth family of   $\T$ invariant Sasaki-Ricci solitons $(\zeta, \eta_{\zeta},  g_{\zeta})$ on $M$ parametrized by the Reeb vector field  $\zeta\in\cU$. Moreover, $(\eta_{\xi_0}, g_{\xi_0})=(\eta_0, g_0)$, and for all $\zeta\in \cU$,  $(\zeta, \eta_{\zeta}, g_{\zeta})$ is a simple deformation of $(\xi_0, \eta_0, g_0)$  with respect to $\T$, and has maximal torus-symmetry $\T$, i.e. $\T$ is a maximal connected torus in $\text{Aut}(\xi_0, \eta_0, g_0)$. 
\end{thm}

\begin{proof} 
 By a Type-I deformation(c.f. Lemma \ref{l-2-6}),  there is a small neighborhood $\cV$ of $\xi_0$ in $\mathcal H$, and  a smooth family of $\T$-invariant Sasaki metrics $g_{\xi}$ on $M$ with Reeb vector field $\xi\in\cV$, such that $g_{\xi_0}=g_0$ is the original Sasaki-Ricci soliton.  
 Fix a large integer $k$.  We denote by $L^2_{k}$ the space of $\T$ invariant $L^2_{k}$ real valued functions on $M$, and denote by  $W_{k,\xi}$ the space of $\T$ invariant $L^2_k$ Hamiltonian holomorphic potentials of  $\omega_{\xi}$. Let $V_{k,\xi}$  be the orthogonal complement of $W_{k,\xi}$ with respect to the $L^2$ inner product defined using volume form $e^{-h_{\xi}}d\mu_{\xi}$.  Here $h_{\xi}$ is the normalized  Ricci potential of $\omega_{\xi}$, i.e. $$Ric^T(\omega_{\xi})+\sqrt{-1}\p_{B,\xi}\bp_{B,\xi} h_{\xi}=(2n+2)\omega^T_{\xi}, $$
 and $\int_M e^{-h_{\xi}}d\mu_{\xi}=1$. 
Now we choose a small open neighborhood $V'_{k+2, \xi_0}$ of $0$ in $V_{k+2, \xi_0}$, and make $\cV$ smaller if necessary, so that the transverse K\"ahler deformation $\omega_{\xi, \phi}^T=\omega_{\xi}^T+\p_{B, \xi}\bp_{B,\xi}\phi$ is still transverse K\"ahler for all $\xi\in \cV$ and $\phi\in V'_{k+2, \xi_0}$. 
Then we define a map $$F: \cV\times V'_{k+2,\xi_0}\times \ft  \rightarrow L^2_k; (\xi, \phi, Y)\mapsto \log\frac{(\omega^T_{\xi,\phi})^n}{(\omega_{\xi}^T)^n}+(2n+2)\phi+h_{\xi}+\eta_{\xi, \phi}(Y+Y_0), $$
 where $\eta_{\xi,\phi}(Y+Y_0)$ is the Hamiltonian holomorphic potential  with respect to $\omega_{\xi, \phi}$ corresponding to $Y+Y_0\in\ft$. The zeroes of the map $F$ are $\T$ invariant  Sasaki-Ricci solitons with Ricci potential in $\ft$.
 The differential of $F$ at $(\xi_0,0,0)$ along the last two components is given by $$P(\phi, Y)=\frac{1}{2}L_{\omega_0}\phi+\eta_0(Y).$$
 By Section \ref{S-2-3}, $P$ is an isomorphism.
 Thus by the implicit function theorem there is smaller neighborhood $\cU\subset\cV$,  such for any $|\xi-\xi_0|$ small there is a $(\phi, Y)$ small depending smoothly on $\xi$ such that $\omega_{\xi,\phi}$ is a $\T$-invariant Sasaki-Ricci soliton with Reeb vector field $\xi$. The fact that $\phi$ is indeed smooth follows from the standard elliptic regularity applied to the local Sasaki-Ricci soliton equation.   Notice $\omega_{\xi, \phi}$ is a transverse K\"ahler deformation of $\omega_{\xi}$, so is a simple deformation of $(\xi_0, \eta_0, g_0)$ with respect to $\T$.  Finally we can choose $\cU$ small so that $\T$ is a maximal torus  in $\Aut(\zeta, \eta_\zeta, g_\zeta)$ for $\zeta\in \cU$. Otherwise there is a sequence $\zeta_i\rightarrow \xi_0$ so that $\Aut(\zeta_i, \eta_{\zeta_i}, g_{\zeta_i})$ contains a connected tori $\T'_i$ strictly bigger than $\T$, by taking limit we see $\Aut(\xi_0,\eta_0, g_0)$ contains a connected tori $\T^{'}$ which is strictly bigger than $\T$.   This contradicts the fact that $\T$ is a maximal tori in $\Aut(\xi_0,\eta_0, g_0)$. 
 
 \end{proof}

\subsection{Rigidity of Sasaki manifolds with positive transverse bisectional curvature}\label{S-4-2}
 In this subsection we prove the local rigidity of Sasaki structures with positive transverse bisectional curvature. As we will see, this rigidity can not hold in the strict sense as in the K\"ahler case, due to the possible Type-I deformations. This will also make the proof slightly more complicated, and the main issue is that the space of basic quantities does not behave well under Type-I deformation. For this reason we will appeal to a quantitative implicit function theorem (c.f. Lemma \ref{IFT}, \ref{IFT1}). 
 
 For the purpose of doing infinite dimensional analysis, it is convenient to define the appropriate Banach space topology. In this section we will always fix a smooth background metric $\hat g$ on $M$,  then we define the $C^{k, \alpha}$ difference between any two Sasaki structures $(\xi_1, \eta_1, g_1)$ and $(\xi_2, \eta_2, g_2)$ to be $||\xi_1-\xi_2||_{C^{k, \alpha}}+||\eta_1-\eta_2||_{C^{k, \alpha}}+||g_1-g_2||_{C^{k, \alpha}}$, where the norms are measured with respect to $\hat g$. Similarly one can define other norms for other quantities. The important thing here is that we are measuring everything by the fixed metric $\hat g$.  We say $(\xi_1, \eta_1, g_1)$ is $\delta$ close to $(\xi_2, \eta_2, g_2)$ in $C^{k, \alpha}$ if their $C^{k, \alpha}$ difference is at most $\delta$.    From now on in this section we fix a large integer $k$ and $\alpha\in (0,1)$.  
 
 Now we fix a smooth Sasaki structure  $(\xi_0, \eta_0, g_0)$ on $M$, and a maximal torus $\T$ in $\Aut(M, \xi_0, \eta_0, g_0)$, and thus a maximal abelian Lie subalgebra $\ft\subset \g$ so that $\xi_0\in \ft$.  Since $\g$ is a finite dimensional vector space,  for a fixed $(k', \alpha')$, the $C^{k', \alpha'}$ norm on $\g$ (and $\ft$) is uniformly equivalent to any other fixed norm.

For  simplicity of  presentation, we introduce the notation $\sharp=\sharp[*]$, meaning that $\sharp$ is a  monotonically increasing function of $*$, which is defined for $*$ positive and small, and $\sharp$ tends to zero as $*$ tends to zero.  
 
 \begin{thm}\label{rigidity}
Suppose  $(M, \xi_0, \eta_0, g_0)$ has positive transverse bisectional curvature. Then there is a  number $\delta_0>0$, and a function $\epsilon=\epsilon[\delta]$ defined for $\delta\in (0, \delta_0]$, so that for any Sasaki structure $(\xi, \eta, g)$ on $M$ that is $\delta$ close to $(\xi_0, \eta_0, g_0)$ in $C^{k, \alpha}$, there is a $C^{k-4, \alpha}$ diffeomorphism $F$ of $M$ so that $F^*(\xi,\eta, g)$ is a simple deformation of $(\xi_0, \eta_0, g_0)$ with respect to $\T$, and is $\epsilon$ close to $(\xi_0, \eta_0, g_0)$ in $C^{k-5, \alpha}$. 
\end{thm}

The remainder of this section is devoted to the proof of this theorem.  
We will take a small ball $B=B_{\epsilon}(\xi_0)\subset \ft$ (with respect to a fixed norm on $\ft$) around $\xi_0$, where $\epsilon$ is a small positive number to be specified below. 
Since $\ft$ is a finite dimensional vector space, we can parametrize the ball $B$ by an Euclidean ball $B_\epsilon(0)$ of radius $\epsilon$: for any $s\in B_\epsilon(0)$, there is a unique $\xi_s\in B$ such that $\|\xi_s-\xi_0\|_{C^{k, \alpha}}$ is controlled by $|s|$. 
For any $\xi_s\in B$, we denote $(\xi_s, \eta_s, g_s)$ to be the Type-I deformation of $(\xi_0, \eta_0, g_0)$ with respect to $\T$. These are all smooth. By \eqref{e-2-7}, $(\xi_s, \eta_s, g_s)$ is $\epsilon_1=\epsilon_1[\epsilon]$-close to $(\xi_0, \eta_0, g_0)$ in $C^{k, \alpha}$ norm if $\epsilon$ is sufficiently small; in particular  $(\xi_s, \eta_s, g_s)$ has positive transverse bisectional curvature.

The proof of Theorem \ref{rigidity} consists of two steps. 
In step one, by possibly applying a diffeomorphism, we transform $(\xi, \eta, g)$ to one such that $\xi=\xi_s\in B$ and $\eta=\eta_s$ for some $s\in B_\epsilon(0)$, and $g$ is close to $g_s$.  In step two, we show that for any $s$, any Sasaki metric $g$ compatible with $(\xi_s, \eta_s)$ which is close to $(\xi_s, \eta_s, g_s)$ is isomorphic to a transverse K\"ahler deformation of $(\xi_s, \eta_s, g_s)$. The second step is similar to the rigidity theorem in the K\"ahler setting, but is slightly more complicated since we need uniformity in $s$.  We start with step one,  

\begin{prop}\label{P-canonical}
There is a function $\underline\delta=\underline\delta[\delta]$ with the following property. 
Suppose $(\xi, \eta, g)$ is $\delta$ close to $(\xi_0, \eta_0, g_0)$ in $C^{k, \alpha}$ norm for $\delta$ sufficiently small. Then modifying by a $C^{k-3}$ diffeomorphism if necessary, we may assume  $\xi=\xi_s\in B$, $\eta=\eta_s$ for some $s$ with $|s|\leq \underline \delta$, and  $||g-g_s||_{C^{k-4}}\leq \underline\delta$.  \end{prop}

First an application of a theorem by Grove-Karcher-Ruh (\cite{GrKa}, \cite{GrKaRu}, \cite{Kim}) gives rise to the following lemma.

\begin{lem}\label{Group identification}
There is a function $\delta_1=\delta_1[\delta]$  such that for any Sasaki structure $(\xi, \eta, g)$ that is $\delta$ close to $(\xi_0, \eta_0, g_0)$ in $C^k$, there is a $C^{k-1}$ diffeomorphism  $f$ of $M$ that is $\delta_1$ close to the identity map in $C^{k-1}$, so that viewed as subgroups of the group of  smooth diffeomorphisms,  we have a natural inclusion
$$\Aut(f^*\xi, f^*\eta, f^*g)=f^{-1}\circ \Aut(\xi, \eta, g)\circ f \hookrightarrow \Aut(\xi_0, \eta_0, g_0).$$
\end{lem}
\begin{proof}
This follows from the arguments in \cite{Kim}. For the convenience of readers we include a sketch of proof here. First we notice that  the automorphism group of a compact Sasaki manifold is always compact. Let $(M, \xi_0, \eta_0, g_0)$ be a compact Sasaki manifold. It follows from the compactness that for any $\zeta>0$, there is a $C^k$ neighborhood $\cU$ of $(\xi_0, \eta_0, g_0)$ in the space of all Sasaki structures  on $M$, such that for any $(\xi, \eta, g)$ in $\cU$,  $\Aut(\xi,\eta, g)$ is within the $\zeta/3$ neighborhood of $\Aut(\xi_0, \eta_0, g_0)$ in the $C^{k-1}$ topology in $\Diff(M)$, i.e. for any $q\in \Aut(\xi, \eta, g)$, there is a $q_0\in \Aut(\xi_0, \eta_0, g_0)$ such that $d_{C^{k-1}}(q, q_0)\leq \zeta/3$. 
By choosing a $\zeta/3$ dense net of $\Aut(\xi_0, \eta_0, g_0)$, we can define a measurable map $P: \Aut(\xi, \eta, g)\rightarrow \Aut(\xi_0, \eta_0, g_0)$, such that for any $q\in \Aut(\xi, \eta, g)$, we have $d_{C^{k-1}}(P(q),q)\leq \zeta$.  It follows that there is a constant $C_1>0$ independent of $\zeta$ so that for any $q_1, q_2\in \Aut(\xi, \eta, g)$,
$$d_{C^{k-1}}(P(q_1 q_2) \circ P(q_2)^{-1}, P(q_1))\leq C_1\zeta.$$
Hence $P$ is an almost homomorphism, in the sense of \cite{GrKaRu}.  Using the notion of center of mass for maps, it is proved in \cite{GrKaRu} that for $\epsilon$ sufficiently small, there is a measurable homomorphism $Q: \Aut(\xi, \eta, g) \rightarrow \Aut(\xi_0, \eta_0, g_0)$ with $d_{C^{k-1}}(P(q), Q(q))\leq C_2\zeta$ for some constant $C_2>C_1$ and all $q\in \Aut(\xi, \eta, g)$. 
Then by general Lie group theory,  $Q$ is indeed a Lie group homomorphism. For any $q\in \Aut(\xi, \eta, g)$ with $Q(q)=\id$,  the property of $Q$ ensures that $d_{C^{k-1}}(q, \id)\leq (C_2+1)\epsilon$. By Corollary 2.5 in \cite{GrKa}, $Q$ must be injective, and hence a Lie group embedding. Thus we obtain two  actions of $\Aut(\xi, \eta, g)$ on $M$ which are $C^{k-1}$ close. By the stability theorem for group actions \cite{GrKa}, there is a diffeomorphism $f$  which is $C^{k-1}$ close to the identity and  conjugates these two actions. Theorem A in \cite{GrKa} is only stated for $C^1$ topology, but it is straightforward to extend it to our case. Here we  emphasize that from the proof we only obtain $C^{k-1}$ regularity on the  conjugating map $f$, but elements in $\Aut(f^*\xi, f^*\eta, f^*g)$ are indeed smooth, since they belong to the larger group $\Aut(\xi_0, \eta_0, g_0).$\end{proof}

\begin{proof}[Proof of Proposition \ref{P-canonical}]
Suppose $(\xi, \eta, g)$ is $\delta$ close to $(\xi_0, \eta_0, g_0)$ in $C^{k, \alpha}$.
By Lemma \ref{Group identification},  after applying a $C^{k-1}$ diffeomorphism $f$ if necessary, we can assume $\xi\in\g$, the Lie algebra  of $\Aut(\xi_0, \eta_0, g_0)$. Moreover $f$ is $\delta_1$ close to identity,  so we can assume $(\xi, \eta, g)$ is $\delta_2=\delta_2[\delta]$ close to $(\xi_0, \eta_0, g_0)$ in $C^{k-2}$.  Since $\ft$ is a maximal abelian sub algebra of $\g$, there is an element $S$ in $\Aut(\xi_0, \eta_0, g_0)$ such that $Ad_S \xi$ lies in $\ft$.  Since $\xi$ is $\delta_2$ close to $\xi_0$ in $C^{k-2}$ and $\Aut(\xi_0, \eta_0, g_0)$ is finite dimensional, $S$ can be chosen to be $\delta_3=\delta_3[\delta]$ close to the identity in $\Aut(\xi_0, \eta_0, g_0)$ in $C^{k-1}$. Conjugating the Sasaki structure $(\xi, \eta, g)$ by $S$, we may assume $\xi\in \ft$, and $(\xi, \eta, g)$ is $\delta_4=\delta_4[\delta]$ close to $(\xi_0, \eta_0, g_0)$ in $C^{k-2}$. Again since $\ft$ is finite dimensional,  the $C^{k, \alpha}$ norm is uniformly equivalent to any fixed norm. In particular, we can assume $\xi=\xi_s\in B$ for some $s$, and then $(\xi, \eta, g)$ is $\delta_5=\delta_5[\delta]$ close to $(\xi_s, \eta_s, g_s)$ in $C^{k-2}$. 
 
Next we apply Gray's stability theorem \cite{Gr}  to obtain  an $\xi_s$-invariant diffeomorphism $F$ that is $\delta_6=\delta_6[\delta]$ close to identity in $C^{k-3}$ such that $F^*\eta=\eta_s$.  Using the Moser's trick as in the proof of \cite{Geiges} Theorem 2.2.2 to the path $\eta_t=(1-t)\eta_s+t\eta$, we  need to solve the following equation for $X_t$ (compare with equation (2.2) in \cite{Geiges}, page 61),
\[
\dot \eta_t+\iota_{X_t} d\eta_t=\mu_t \eta_t,
\]
where $\mu_t$ is a suitable function. 
Notice that $\eta$ and $\eta_s$ have the same Reeb vector field $\xi=\xi_s$. It follows that $\eta_t$ has the Reeb vector field $\xi=\xi_s$ for all $t$. Plugging in the Reeb vector field $\xi=\xi_s$ we get $\mu_t=0$ and $X_t$ is then characterized by
\[
\iota_{X_t} d\eta_t=\eta_s-\eta. 
\]
This equation has a unique solution $X_t$ if we require $X_t\in \Ker(\eta_t)$. 
Clearly $[X_t, \xi_s]=0$ since $\eta_t, d\eta_t$ are $\xi_s$-invariant. Moreover, $\|X_t\|_{C^{k-2}}\leq C\|\eta-\eta_s\|_{C^{k-2}}$ for some uniform constant $C$.  We then define
$F_t$  to be the flow of $X_t$ and it satisfies $F_t^*\eta_t=\eta_s$. Take $F=F_1$, then $F^{*}(\xi, \eta, g)=(\xi_s, \eta_s, F^{*}g)$ is $\delta_6=\delta_6[\delta]$-close to $(\xi_s, \eta_s, g_s)$ in  $C^{k-4}$.
\end{proof}

With Proposition \ref{P-canonical}, the proof of Theorem \ref{rigidity} reduces to a uniform  local rigidity theorem around $(\xi_s, \eta_s, g_s)$ for any $\xi_s\in B$.

 \begin{prop} \label{S-8-2}
 There is a number $\epsilon>0$, and a function $\tau=\tau[\delta]$ defined for $\delta\in (0, \epsilon]$ with the following effect. For $\xi_s\in B=B_\epsilon(\xi_0)$, any Sasaki structure  $(\xi_s, \eta_s, g)$ that is $\delta$-close to $(\xi_s, \eta_s, g_s)$ in $C^{k, \alpha}$,  there is a $C^{k+1, \alpha}$ diffeomorphism $F$ of $M$ so that $F^*(\xi_s, \eta_s, g)$ is a transverse K\"ahler deformation of $(\xi_s, \eta_s, g_s)$ by a potential whose $C^{k+2, \alpha}$ norm is at most $\tau$.  
 \end{prop}

For a fixed $s$ (say $s=0$), this is a standard extension of the corresponding fact in the K\"ahler setting. For the purpose of exposition we first write down the details of the proof in this case.
Recall that, any Sasaki structure on $M$  of the form  $(\xi_0,\eta_0, g)$ is induced by an integrable almost CR structure $(\cD_0=\Ker(\eta_0), \Phi)$ compatible with $(\xi_0, \eta_0)$, and the metric is given by, 
\begin{equation}\label{E4.2-k}
g=\eta_0\otimes \eta_0+\frac{1}{2}d\eta_0 (\I\otimes \Phi).
\end{equation}
 Denote by $\cJ$ the space of  all $C^{k, \alpha}$ almost CR structures compatible with $\eta_0$. For any $\Phi\in \cJ$, it induces a metric (called ``K-contact" metric in literature) by \eqref{E4.2-k}; the only difference with a Sasaki structure is that we do not require $\Phi$ to be integrable. Nevertheless, by \eqref{E4.2-k},  $\|\Phi_1- \Phi_2\|_{C^{k, \alpha}}$ (for any $\Phi_1, \Phi_2\in \cJ$) is uniformly equivalent to $\|g_1-g_2\|_{C^{k, \alpha}}$.

Denote by $\Phi_0$  the integrable  almost CR structure defined by $g_0$. Then $\Phi_0$ determines a decomposition $\cD_0\otimes \C=T_B^{1,0}\oplus T_B^{0,1}$, where $T^{1, 0}_B$ is the $\i$ eigenspace of $\Phi_0$. Any $\Phi$ that is $C^{k, \alpha}$ close to $\Phi_0$ also determines a nearby decomposition, and so corresponds to a unique element $\mu$ in $\Omega_{k, \alpha}^{0, 1}(T_B^{1,0})$.  Here $\Omega_{k, \alpha}^{p,q}(T_B^{1,0})$ denotes the space of $T_B^{1,0}$-valued $(p,q)$ forms that is of class $C^{k, \alpha}$  and that is symmetric with respect to the metric $g_0$, i.e. those which vanish under skew-symmetrization. Since this correspondence is pointwise, it  can indeed be viewed as a smooth local coordinate chart for $\cJ$ near $\Phi_0$. The integrability condition is given by the usual Maurer-Cartan equation
 \begin{equation} \label{integrability}
 \bp_B \mu+[\mu, \mu]=0. 
 \end{equation}
The deformation theory of integrable almost CR structures in $\cJ$ is governed by the following elliptic complex (similar to the K\"ahler case considered by Fujiki-Schumacher \cite{FS})
 \[0\longrightarrow C^{k+2, \alpha}_B(M;\C)\stackrel{\cD_B}{\longrightarrow}T_{\Phi_0}\J=\Omega_{k, \alpha}^{0,1}(T_B^{1,0})\stackrel{\bp_B}{\longrightarrow
}\Omega_{k-1, \alpha}^{0,2}(T_B^{1,0})\longrightarrow \cdots \]
Here $\bp_B$ and $\cD_B$ are the analogue of the Cauchy-Riemann operator and the Lichnerowicz operator for basic quantities, and $C^{k+2, \alpha}_B(M;\C)$ denotes the space of complex-valued basic functions that have \emph{average zero.}

Let $\square_B=\cD_B\cD_B^*+(\bp_B^*\bp_B)^2$.  It is  an elliptic operator, due to the $\xi_0$-invariance of the whole complex.  We define $\mathbb H$ to be the kernel of $\square_B$, and by standard elliptic theory it consists of smooth elements. Now we have

\begin{lem}\label{cohomology}
If $(\xi_0, \eta_0, g_0)$ has positive transverse bisectional curvature, then  $\mathbb H=\{0\}$. In particular $\square_B$ is invertible, and there exists a constant $C_0>0$ such that for any $\mu\in \Omega_{k, \alpha}^{0,1}(T_{B}^{1,0})$,
\begin{equation}\label{E4.2-1}
\|\mu\|_{C^{k, \alpha}}\leq C_0 \|\square_B \mu\|_{C^{k-4, \alpha}}. 
\end{equation}
\end{lem}

\begin{proof} We prove this lemma via a Bochner technique.   Fix any transverse holomorphic coordinates $\{z^i\}$, by assumption, for any two nonzero vectors $u=u^i\p_{z^i}$, $v=v^i\p_{z^i}$ 
$$R^T_{i\bar{j}k\bar{l}} u^iu^{\bar j} v^k v^{\bar l}>0.$$
By an algebraic manipulation (see \cite{Demailly} Proposition 10.14),
for any nonzero symmetric tensor $\mu=\mu^{ij} \p_{z^i}\otimes \p_{z^j}\in \text{Sym}(T_B^{1,0}\otimes T_B^{1,0})$ we have
$$R^T_{i\bar{j}k\bar{l}}\mu^{ik}\mu^{\bar j\bar l}+R^T_{i\bar{j}}\mu^{ik}\mu^{\bar k\bar j}>0.$$ 
Denote by $\Delta_B=\bp_B\bp_B^*+\bp_B^*\bp_B$ the usual basic Laplacian operator on $\Omega^{0,1}_B(T_B^{1,0})$. This should not be confused with the operator $\square_B$ we define above, associated to the Lichnerowicz operator. 
We compute, 
\begin{eqnarray*}
\Delta_B\mu
&=&\bp_B\bp_B^*\mu+\bp_B^*\bp_B\mu
\\&=&-\mu_{\bar{i}\bar{k}, k\bar{j}}-\mu_{\bar{i} \bar{j}, k\bar k}+\mu_{\bar{k}\bar{j}, \bar{i}k}
\\&=&R^T_{k\bar{j}l\bar{i}}\mu_{\bar{l}\bar{k}}+R^T_{k\bar{j}}\mu_{\bar{i}\bar{l}}-\mu_{\bar{i}\bar{j}, k\bar{k}}.
\end{eqnarray*}
For any $\mu\in \Gamma(M, \text{Sym}(T_B^{1,0}\otimes T_B^{1,0}))\cong \Omega_B^{0,1}(M,T_B^{1,0})$ with $\Delta_B\mu=0$, we have
$$R^T_{k\bar{j}l\bar{i}}\mu_{\bar{l}\bar{k}}+R^T_{k\bar{j}}\mu_{\bar{i}\bar{k}}-\mu_{\bar{i}\bar{j}, k\bar{k}}=0.$$
This clearly implies $\mu=0$. Now we show $\mathbb H=0$. This is equivalent to showing that any $\mu\in \Omega_B^{0,1}(M, T_B^{1,0})$ with $\bp_B\mu=0$ is equal to $\cD_B f$ for some complex-valued basic function $f$. By Hodge theory (for basic forms), what we have shown is that $\mu=\bp_B \alpha$ for some $\alpha\in \Omega^0_B(M, T_B^{1,0})$. In other words, $\mu_{\bar{i}\bar{j}}=\alpha_{\bar{i}, \bar{j}}$. $\mu_{\bar{i}\bar{j}}=\mu_{\bar{j}\bar{i}}$ is equivalent to $\bp_B\alpha=0$. By the positivity assumption, we have $H^1_B(M;\C)=0$ (see \cite{He} Proposition 9.4), thus there is a basic function $f$ such that $\alpha=\bp_B f$. Then $\mu=\cD_B f$. This completes the proof. 
\end{proof}

Next  we need a slice theorem for the effect of transverse K\"ahler deformation on $\cJ$. We have the decomposition 
\[\Omega_{k, \alpha}^{0,1}(T_{B}^{1,0})=\text{Im} \cD_{B} \oplus \Ker \cD_{B}^*; \ \ C^{k+2, \alpha}_B(M;\C)=\Ker \cD_B\oplus \text{Im} \cD_B^* \]
Let $\cG$ be the group of $C^{k+1, \alpha}$ strict contact transformations of $(M, \eta_0)$. 
Formally the Lie algebra of $\cG$ is $C^{k+2, \alpha}_B(M;\R)$ together with the natural Poisson algebra structure. More precisely, an element $\phi\in C^{k+2, \alpha}_B(M;\R)$ gives rise to a $C^{k+1, \alpha}$ strict contact vector field $X_\phi$ through the formula  $d\phi=\frac{1}{2}d\eta(X_\phi, \cdot)$. 
 The operator $\cD_B$ sends $\phi$ to $\bp_B(X_\phi)$, and this obviously extends $\C$-linearly to the complexification $C^{k+2, \alpha}_B(M;\C)$. The geometric meaning of the action on  imaginary functions is through the transverse K\"ahler deformation in a fixed basic class, much as the similar picture in the K\"ahler setting.  Any transverse K\"ahler deformation $\eta_t=\eta_0+d^c_B \phi_t$ gives rise, by a Moser type theorem, to an isotopy $f_t$ such that $f_t^* \eta_t=\eta_0$, and $\frac{d}{dt}|_{t=0}f_t^*\Phi=\sqrt{-1}\cD_B\dot{\phi}_0$. For our purpose we need to fix the gauge with respect to the $\cG^\C$ action but the group $\cG^\C$ is indeed not defined.   

We need some preparations. 
Denote by $C^{k+1, \alpha}_B(M, TM)$ the space of all $C^{k+1, \alpha}$ vector fields on $M$ that commute with $\xi_0$. For any $Y\in C^{k+1, \alpha}_B(M, TM)$, we define $F_Y: M\rightarrow M$ by setting $F_Y(x)=\exp_x(Y_x)$, where the exponential map is with respect to the metric $g_0$. 
For $\sigma>0$ small we consider a neighborhood $N_\sigma$ in $C^{k+1, \alpha}_B(M, TM)$ consisting of elements $Y$ with $\|Y\|_{C^{k+1, \alpha}}<\sigma$. 
Now we define a map $\Sigma: N_\sigma \rightarrow Im \p_B\subset \Omega^{2,0}_{k, \alpha, B}(M;\C)$ by
\begin{equation}\label{E4.2-sigma}
\Sigma(Y)=\p_B \gamma^{1, 0},\; \mbox{with}\; \gamma=(F_Y^{-1})^*\eta_0-\eta_0. 
\end{equation}
\begin{lem} \label{smoothness}
For $\sigma>0$ sufficiently small, $F_Y$ is a $C^{k+1, \alpha}$ diffeomorphism of $M$ for all $Y\in N_\sigma$. Moreover,   $Y\mapsto F_Y$ viewed as a map from $N_\sigma$ to $Map^{k+1, \alpha}(M, M)$ (the space of $C^{k+1, \alpha}$ maps from $M$ to itself) is a smooth map between Banach manifolds.  As a consequence, $\Sigma$ is also a well-defined smooth map. 
\end{lem}

\begin{proof}
The key point is that our definition of $F_Y$ depends pointwise on $Y$.  The standard ODE theory ensures that the exponential map $\exp$ (with respect to $g_0$),  viewed as a map from $TM$ to $M$, is smooth.  We view $Y$ as a $C^{k+1, \alpha}$ map from $M$ to $TM$, then $F_Y$ is the composition of $\exp$ with $Y$, so is in $Map^{k+1, \alpha}(M, M)$.   Similarly one can show $F_Y$ depends smoothly on $Y$. For  $Y, Z$ in $N_\sigma$
we have 
$$F_Z(x)-F_Y(x)=\int_0^1 d(\exp_x)|_{(1-t)Y(x)+tZ(x)}(Z(x)-Y(x))dt, $$
where the difference in the left hand side is understood in the sense that when $\sigma$ is small we may assume for any fixed $x$,  $F_Y(x)$ stays in a fixed coordinate chart in $M$ for all $Y\in N_\sigma$.
Viewing $d\exp$ as a smooth map from $TM\otimes TM$ to $M$ and $((1-t)Y+tZ, Z-Y)$ as a $C^{k+1, \alpha}$ map from $M$ to $TM\otimes TM$,   it is easy to see that $F_Y$ depends $C^1$ on $Y$, and its $C^1$ norm depends only on $\sigma$.  Similarly all higher order derivatives also depend only on the order and $\sigma$.  Notice if $Y=0$, then $F_Y$ is the identity map, so if we choose $\sigma$ small, we can also ensure $F_Y$ is a $C^{k+1, \alpha}$ diffeomorphism for all $Y\in N_\sigma$. 
 \end{proof}
 
 Then we have
 
 \begin{lem} \label{P}For $\sigma>0$ small enough, $\Sigma^{-1}(0)\cap N_\sigma$ is a smooth submanifold. Furthermore, there is a small neighborhood $W$ of zero in $C^{k+2,\alpha}_B(M;\C)$, and 
 a smooth diffeomorphism
 \begin{equation}\label{E4.2-p}P: W\rightarrow \Sigma^{-1}(0)\cap N_\sigma\end{equation} 
 such that $dP|_0: \phi_1+\sqrt{-1}\phi_2\mapsto X_{\phi_1}+\Phi_0 X_{\phi_2}$. Moreover we have ${(F_{P(\phi)}})^*\Phi_0\in \J$ for all $\phi\in W$. 
 \end{lem}

\begin{proof}
The tangent map $d\Sigma$ at $0\in N_\sigma$ is given by \[d\Sigma|_0(Y)=-\p_B\left(L_Y\eta_0\right)^{1, 0}=-\p_B\left(\iota_Y d\eta_0\right)^{1, 0}.\] Then $Y\in \Ker (d\Sigma|_0)$ if and only if  $(\iota_Yd\eta_0)^{1, 0}$ is a $\p_B$-closed basic $(1, 0)$-form. Since the basic cohomology $H^1_B=0$, it follows that $-(\iota_Y d\eta_0)^{1, 0}=\p_B \phi$ for some $C^{k+2, \alpha}$ complex-valued basic function $\phi$. For $\phi=\phi_1+\i \phi_2\in  C^{k+2, \alpha}_{B}(M;\C)$, denote $Y_{\phi}=X_{\phi_1}+\Phi_0 X_{\phi_2}$, where $X_{\phi_i}$ is defined by $d\phi_i=\frac{1}{2}d\eta_0(X_{\phi_i}, \cdot)$ with $i=1, 2$. 
Hence we have 
\begin{equation}\label{E4.2-ker}\Ker (d\Sigma|_0)=\left\{Y_\phi|\phi \in C^{k+2, \alpha}_{B}(M;\C)\right\}\simeq C^{k+2, \alpha}_{B}(M;\C). \end{equation}
Now $d\Sigma|_0$ is surjective, and we can write down a bounded right inverse $\hat \Sigma_R^{-1}$: given $\beta\in Im\p_B\subset \Omega^{2,0}_{k,\alpha, B}(M;\C)$,  we define
\begin{equation}\label{E4.2-inverse}
\hat \Sigma_R^{-1}(\beta)=Y\; \mbox{with}\; \left(\iota_Y d\eta_0\right)^{1,0}=\Delta_{\p_B}^{-1}\p_B^*\beta.\end{equation} 
Note that $\Delta_{\p_B}$ is a positive operator on $\Omega^{1, 0}_{k+1, \alpha, B}$ since $H^1_B=0$. We consider the map
\[
Q: N_\sigma\rightarrow Im(\p_B)\oplus \Ker(d\Sigma|_0)\; \mbox{with}\; Q(Y)=\left(\Sigma(Y), Y-\hat \Sigma^{-1}_R d\Sigma|_0(Y)\right)
\]
It is ready to see that $Q(0)=0$, and $dQ|_0$ is invertible. By the implicit function theorem and  by \eqref{E4.2-ker}, we can then define $P(\phi)=Q^{-1}(0, Y_\phi)$ 
for any $\phi\in C^{k+2, \alpha}_{B}(M;\C)$ in a small neighborhood $W$ of zero.  By making $W$ and $\sigma$ small we may assume $P$ is a diffeomorphism onto $\Sigma^{-1}(0)\cap N_\sigma$. It is also easy to check that
\[
dP|_0: \phi_1+\sqrt{-1}\phi_2\mapsto X_{\phi_1}+\Phi_0 X_{\phi_2}
\] 
Note that by definition of $P$, $\p_B ((F^{-1}_{P(\phi)})^*\eta_0-\eta_0)^{1, 0}=0$. If $\phi\in W$, then $(F^{-1}_{P(\phi)})^{*}d\eta_0$ is a basic $(1, 1)$-form. Hence  $(F^{-1}_{P(\phi)})^*d\eta_0$ is compatible with $\Phi_0$ and  is indeed a transverse K\"ahler deformation of $(\xi_0, \eta_0, g_0)$; in other words, ${(F_{P(\phi)})}^*\Phi_0$ is compatible with $d\eta_0$, i.e. ${(F_{P(\phi)}})^*\Phi_0\in \J$.  
\end{proof}

With the discussion above,  we are ready to prove Proposition \ref{S-8-2} for $s=0$. 

\begin{proof}[Proof of Proposition \ref{S-8-2} when $s=0$]
 Denote by $\cU_\tau$ the $\tau$-neighborhood of zero in $\Ker(\cD_B^*)\subset \Omega^{0,1}_{k, \alpha}(T_B^{1,0})$  and by $\cV_{\tau}$ the $\tau$-neighborhood of zero in $C^{k+2,\alpha}_B(M;\C)$.
We define a smooth map 
\[
R: \cU_\tau\times \cV_\tau\rightarrow \Omega_{k, \alpha}^{0,1}(T_{B}^{1,0})\;\mbox{with}\; R(\mu, \phi)={(F_{P(\phi)})}^{*}\Phi_0+\mu. 
\]
 Here the addition is understood in the sense that we have chosen the coordinate chart that identifies a neighborhood of $\Phi_0$ in $\cJ$ with a neighborhood of zero in $\Omega^{0, 1}_{k, \alpha}(T_B^{1,0})$. $R$ is smooth since $F, P$ are both smooth, and $\Phi_0$ is a smooth almost CR structure. 
One can  compute the differential of $R$ at $(0, 0)$
\begin{equation}
dR|_{(0, 0)}(\nu, \psi)=\nu+L_{dP(\psi)} \Phi_0=\nu+\cD_B\psi. 
\end{equation}
This implies that $dR_{(0, 0)}$ is an isomorphism. Hence by implicit function theorem, there is a constant $\delta_0>0$ so that for any $\Phi\in\cJ$ that is $\delta$ close to $\Phi_0$ with $\delta\leq \delta_0$, we can represent $\Phi$ by an element $\nu=R(\mu, \phi)$ for some $(\mu, \phi)\in \cU_\tau\times \cV_\tau$ with $\tau=\tau[\delta]$. Denote $\mu_1={(F_{P(\phi)})}^{*}\Phi_0$ and $\nu=\mu_1+\mu$. Then we further have 
\[
\|\mu_1\|_{C^{k, \alpha}}\leq C_3\|\phi\|_{C^{k+2, \alpha}}
\]
for   some uniform constant $C_3$.
Now if we assume $\Phi$ is integrable, then  we have $\bar\p_B \nu+[\nu, \nu]=0$. Since $\Phi_0$ is integrable, we also have $\bar \p_B\mu_1+[\mu_1, \mu_1]=0$. Then we compute
\[
\bp_B \mu=-2[\mu_1, \mu]-[\mu, \mu]. 
\]
By \eqref{E4.2-1}, we have, 
\begin{equation}
\begin{split}
\|\mu\|_{C^{k, \alpha}}\leq& C_0 \|\square_{B,s}\mu\|_{C^{k-4, \alpha}}
=C_0\|\bp_{B,s}^*\bp_{B,s}\bp_{B, s}^* [\mu, \mu]\|_{C^{k-4, \alpha}}\\
\leq& C_4 \|\mu\|_{C^{k, \alpha}}(\|\mu\|_{C^{k, \alpha}}+\|\mu_1\|_{C^{k, \alpha}}). 
\end{split}
\end{equation}
When $\delta$ is sufficiently small, this implies $\mu=0$. This completes the proof. 
\end{proof}

To completely prove Proposition \ref{S-8-2} we  need to carry out the above discussion uniformly at $(\xi_s, \eta_s, g_s)$ for all $s\in B_\epsilon(0)$, and we will add a subscript $s$ in the notation. 

First we note that the Sasaki structure $(\xi_s, \eta_s, g_s)$ has positive transverse bisectional curvature for any $\xi_s\in B$; hence as in Lemma \ref{cohomology}, we have $\mathbb{H}_s=0$. Moreover, we have the following uniform estimate,
\begin{lem} \label{lem8-5}
There is a uniform constant  $C_5>0$ such that, for any $s\in B_\epsilon(0)$ and for all $\mu\in\Omega_{k, \alpha, s}^{0,1}(T_{B, s}^{1,0})$,
 \[ ||\mu||_{C^{k, \alpha}}\leq C_5||\square_{B, s}\mu||_{C^{k-4, \alpha}}. \] 
\end{lem}
\begin{proof}Clearly $\square_{B, s}$ is a positive operator by the fact that $\mathbb{H}_s=0$. 
Hence if the statement is not true, then there is a sequence $s_i$,  $\mu_i \in\Omega_{k, \alpha, s_i}^{0,1}(T_{B, s_i}^{1,0})$ and positive $\l_i$, with $||\mu_i||_{C^{k, \alpha}}=1$, $\square_{B, s_i} \mu_i=\lambda_i \mu_i$, and $\lambda_i\rightarrow 0 $. By compactness and elliptic regularity we may assume $s_i\rightarrow s_\infty$, and $\mu_i\rightarrow \mu_\infty$ in $C^{k, \alpha}$. Clearly $\mu_\infty \in  \Omega_{k, \alpha, s_\infty}^{0,1}(T_{B, s_\infty}^{1,0})$, and $\square_{B, s_\infty}\mu_\infty=0$.  However $\mathbb H_{s_\infty}=0$ implies $\mu_\infty=0$. This contradicts that $||\mu_i||_{C^{k, \alpha}}=1$. 
\end{proof}

Next we need to show that $\Sigma_s$ and $P_s$, as defined  in \eqref{E4.2-sigma} and \eqref{E4.2-p} behave uniformly.
\begin{lem}
By making $\epsilon$ smaller if necessary, we may assume there is a constant $\sigma>0$ independent of $s$ with $|s|\leq\epsilon$, and smooth maps $P_s$ from  $N_{\sigma, s}$ to $C^{k+1, \alpha}_{B, s}(M, TM)$ with $P_s(0)=0$, $dP_s|_0(\phi_1+i\phi_2)=X_{\phi_1, s}+\Phi_s X_{\phi_2, s}$ and $P_s$ has uniformly bounded Hessian on $N_{\sigma, s}$ independent of $s$. Moreover, we have that for all $\phi\in N_{\sigma, s}$, $F_{P_s(\phi)}^*\Phi_s\in \J_s$. \end{lem}

\begin{proof}The argument is the same  as in Lemma \ref{smoothness} and Lemma \ref{P} except that we need to use a quantitative implicit function theorem (Lemma \ref{IFT1}). It suffices to show that $\Sigma_s$ has a uniformly bounded Hessian in a neighborhood of $0$ with definite size independent of $s$, and $d\Sigma_s|_0$ has a uniformly bounded right inverse. The first condition can be easily checked similar as in the proof of Lemma \ref{smoothness}. By \eqref{E4.2-inverse}, the second condition requires a uniform positive lower bound of $\Delta_{\p_B, s}$ on $\Omega^{1, 0}_{k+1, \alpha, B, s}$. The latter can be proved  using a contradiction argument, similar to the proof of Lemma \ref{lem8-5}. 
\end{proof}

With these prepared, we now start proving a uniform slice theorem. Similar as before,  for $\epsilon$ small, we have for $\tau, |s|\leq \epsilon$ smooth maps 
$$R_s: \cU_{\tau, s} \times \cV_{\tau, s} \rightarrow \Omega_{k, \alpha, s}^{0,1}(T_{B, s}^{1,0}). $$
which sends $(\mu, \phi)$ to $({F_{P_s(\phi)}})^*\Phi_s+\mu$.   The differential of $R_s$ at $(0, 0)$ is given by ${dR_s}|_{(0, 0)}(\nu, \psi)=\nu+\cD_{B, s}\psi$. 

\begin{lem} \label{representation}
 There are a positive constant $\delta'$, and a function $\tau=\tau[\delta]$,  such that if $\delta\leq \delta', |s|\leq \delta'$, then  for any almost CR structure  $\Phi$ that is compatible with $(\xi_s, \eta_s)$, and is $\delta$ close to the one defined by $(\xi_s, \eta_s, g_s)$ in $C^{k, \alpha}$,  there are $\mu_s\in \cU_{\tau, s}$ and $\phi_s\in \cV_{\tau, s}$, such that $\Phi$ is represented by $R_s(\mu_s, \phi_s)$.   
\end{lem}

\begin{proof}
This follows again from the implicit function theorem (Lemma \ref{IFT}).  To apply Lemma \ref{IFT},   we need uniform bounds on the norm of the inverse of $dR_s|_{(0, 0)}$, and on the norm of Hessian of $R_s$. For the Hessian bound of $R_s$,  since $R_s$ is linear on $\mu$ by definition we only need to check this for the map
\[
\phi\rightarrow ({F_{P_s(\phi_s)}})^*\Phi_s.
\]
Similar as in Lemma 4.7, $F$ and $P$ are both smooth map with bounded Hessian, depending only the geometry of the background metric $(\xi_s, \eta_s, g_s)$ and the norm of $\phi_s$. Hence 
the Hessian bound of $R_s$ on $\cU_{\tau, s}\times \cV_{\tau, s}$ is uniformly bounded as long as $\tau$ and $s$ are less than some small but fixed positive number.

To show $dR_s^{-1}|_{(0, 0)}$ is uniformly bounded,  we claim that there exist $s_1, C_5>0,$ such that for any $s$ with $|s|\leq s_1$,  and  $\psi_s \in C^{k+2, \alpha}_{B, s}(M;\C)$, we have 
\begin{equation}\label{uld}
||\cD_{B, s}^*\cD_{B, s} \psi_s||_{C^{k-2, \alpha}}\geq C_5 ||\psi_s||_{C^{k+2, \alpha}}. 
\end{equation}
Given \eqref{uld}, we show there is a uniformly positive lower bound of $dR_s|_{(0,0)}$. First we note that,  there is a uniform constant $C^{'}>0$ independent of $s$,  such that for any $h\in \Omega^{0,1}_{k, \alpha, s}(T^{1,0}_{B, s})$ ($\cD_{B, s}^*$ is a second-order differential operator with smooth coefficients determined by the background metric)
\[
\|\cD_{B, s}^*h\|_{C^{k-2, \alpha}}\leq C^{'}\|h\|_{C^{k, \alpha}}
\]
So applying the above to $h=\nu_s+\cD_{B, s}\psi_s$, we have (by definition $\cD_{B, s}^*\nu_s=0$), 
\[
\|\cD_{B, s}^*\cD_{B, s}\psi_s\|_{C^{k-2, \alpha}}=\|\cD_{B, s}^*\left(\nu_s+\cD_{B, s}\psi_s\right)\|_{C^{k-2, \alpha}}\leq C^{'}\|\nu_s+\cD_{B, s}\psi_s\|_{C^{k, \alpha}}. 
\]
Hence given \eqref{uld}, there is a $C_6=C_5/C^{'}$, 
$$||\nu_s+\cD_{B, s}\psi_s\|_{C^{k, \alpha}}\geq C_6 \|\psi_s||_{C^{k+2, \alpha}}. $$
Thus 
\[
\begin{array}{lcl}
||\nu_s||_{C^{k, \alpha}}&\leq& ||\nu_s+\cD_{B, s}\psi_s\|_{C^{k, \alpha}}+||\cD_{B, s}\psi_s||_{C^{k, \alpha}}\\
&\leq& ||\nu_s+\cD_{B, s}\psi_s\|_{C^{k, \alpha}}+C_7 ||\psi_s||_{C^{k+2, \alpha}}\\
&\leq & (1+C_6^{-1}C_7)  ||\nu_s+\cD_{B, s}\psi_s\|_{C^{k, \alpha}}.
\end{array}
\]
Let $C_8=\frac{1}{2}\min((1+C_6^{-1}C_7)^{-1}, C_6)$, then we obtain
\begin{equation}\label{E-uniform}
\|\nu_s+\cD_{B, s}\psi_s\|_{C^{k, \alpha}}\geq C_8 (\|\nu_s\|_{C^{k, \alpha}}+\|\psi_s\|_{C^{k+2, \alpha}}),
\end{equation}
which is the desired uniform lower bound. 

So it suffices to prove \eqref{uld}.   We again argue by contradiction. Suppose it is not true, then  we may choose a sequence $s_i\rightarrow s_\infty$, an element $\phi_i\in C^{k+2, \alpha}_{B, s_i}(M;\C)$ that has $C^{k+2, \alpha}_{B, s_i}$ norm $1$, and that is $L^2$ orthogonal to $\Ker \cD_{B, s_i}$ (with respect to $g_{s_i}$), such that $\cD_{B, s_i}^*\cD_{B, s_i}\phi_i=\lambda_i  \phi_i$ with $\lambda_i\rightarrow 0$. For simplicity of notation we assume $s_\infty=0$. 
As in the proof of Lemma \ref{lem8-5} we may obtain a smooth limit $\phi_\infty\in C^{k+2, \alpha}_{B, 0}(M;\C)$ with $C^{k+2, \alpha}_{B, 0}$ norm $1$, and $\cD_{B, 0}\phi_\infty=0$. 

Now we  notice that by definition, $\Ker \cD_{B, s_i}= \Lambda_s/\C\langle \xi_s-iJ\xi_s\rangle$, where $\Lambda_s$ is the space of holomorphic vector fields on the K\"ahler cone which commute with $\xi_s$. The complex structure of the K\"ahler cone is fixed, but since $\xi_s$ is varying $\Lambda_s$ does not have constant dimension in general.  Here we need to make use of the limiting procedure in order to draw a contradiction. Let $\T_i$ and $\T_0$ be the closed  subtori of $\T$ generated by $\xi_{s_i}$ and $\xi_0$ respectively. Fix a flat metric on $\T$, then by passing to a subsequence we may assume $\T_i$ converges to a closed subset $\T_\infty$ under the Hausdorff metric. Clearly $\T_\infty$ is also a subgroup of $\T$, so $\T_\infty$ is a Lie subgroup and thus also a compact torus.  It contains $\T_0$ but in general they are not equal.  

We claim for $i$ sufficiently large $\T_i$ is contained in $\T_\infty$. To see this, we choose an appropriate integral basis $\{e_1, \cdots, e_r\}$ of $\ft$ so that the first $d$ elements $e_1, \cdots, e_d$ generate the Lie algebra $\ft_\infty$ of $\T_\infty$. So we may identify $\T$ with the standard $(S^1)^r$ and $\T_\infty$ with $(S^1)^d$. Then if an element $\xi$ is not in $\ft_\infty$, then  we may write $\xi=\xi'+\sum_{j=d+1}^r a_j e_j$ for $\xi'\in \ft_\infty$ and for simplicity we may assume $a_{r}\neq0$, then on the torus $\T_\xi$ generated by $\xi$, there is always an element with the last co-ordinate equal to $e^{i\pi}\in S^1$. This implies the Hausdorff distance between $\T_\xi$ and $\T_\infty$ is bounded below by a fixed computable positive number. So we have proved the claim. 

Since $\phi_i$ is invariant under $\T_i$ by construction, we know $\phi_\infty$ is invariant under $\T_\infty$. So $\phi_\infty$ corresponds to a holomorphic vector field  $Y$ on the K\"ahler cone that commutes with $\ft_\infty$. In particular $Y$ commutes with $\xi_{s_i}$ for $i$ large. Then there is an element $\psi_i\in \Ker \cD_{B, s_i}$ that generates $Y$, and since $(\xi_{s_i}, \eta_{s_i}, g_{s_i})$ converges smoothly to $(\xi_0, \eta_0, g_0)$,   we know  $\psi_i$ also converges smoothly to $\phi_\infty$. On the other hand, by definition $\phi_i$ is $L^2$ orthogonal to $\psi_i$, and this implies $\phi_\infty=0$. Contradiction. 
 \end{proof}

Now we return to the proof of Proposition \ref{S-8-2} and the remaining argument is exactly the same as the case $s=0$. 
We apply the above lemma to the integrable almost CR structure defined $(\xi_s, \eta_s, g)$,  which is represented by $\nu$ in the local coordinate chart.  By the above Lemma we can write $\nu=F_{P_s(\phi)}^*\Phi_s+\mu$ with $\mu\in \cU_{\tau, s}$ and $\phi\in \cV_{\tau, s}$. Let $\mu_1=F_{P_s(\phi)}^*\Phi$.  We have $||\mu_1||_{C^{k, \alpha}}\leq C_{9} ||\phi||_{C^{k+2, \alpha}}$.  Then by the integrability condition (\ref{integrability}) we have $\bp_{B, s} \nu+[\nu, \nu]=0$ and $\bp_{B, s}\mu_1+[\mu_1, \mu_1]=0$.  Then 
$$\bp_{B, s}\mu=-2[\mu_1, \mu]-[\mu, \mu]. $$
Then we apply Lemma \ref{lem8-5}, 
\[
\begin{split}||\mu||_{C^{k, \alpha}}&\leq C_5^{-1} ||\square_{B,s}\mu||_{C^{k-4, \alpha}}=C_5^{-1}||\bp_{B,s}^*\bp_{B,s}\bp_{B, s}^* [\mu, \mu]||_{C^{k-4, \alpha}}\\
&\leq C_{10} ||\mu||_{C^{k, \alpha}}(||\mu||_{C^{k, \alpha}}+||\mu_1||_{C^{k, \alpha}}), \end{split}\]
for a uniform constant $C_{10}$. For $\delta$ and $s$ small (depending only on $C_9, C_{10}$),  this implies $\mu=0$. So we conclude that there exists a positive constant $\delta^{''}$, as long as $\delta, |s|\leq \delta''$, any Sasaki structure  $(\xi_s, \eta_s, g)$ that is $\delta$ close to $(\xi_s, \eta_s, g_s)$ in  $C^{k, \alpha}$ is represented by $R_s(0, \phi)$ for some $\phi\in C^{k+2,\alpha}_{B, s}(M;\C)$ with $|\phi|_{C^{k+2, \alpha}}\leq \tau[\delta]$. Then  $F_{P_s(\phi)}^*(\xi_s, \eta_s, g)$ is a transverse K\"ahler deformation of $(\xi_s, \eta_s, g_s)$. It is also easy to see the potential is bounded by $\tau[\delta]$.  This proves Proposition \ref{S-8-2}. 

\begin{lem}  \label{IFT}
Suppose $E$ and $F$  are two Banach spaces. Let $U$ be an open neighborhood of the origin in $E$. Let $f$ be a smooth map from  $U$  to $F$ with $f(0)=0$. Suppose $L=df|_{0}$ is invertible, then we can choose a smaller neighborhood $U'$ of $U$ such that $f: U'\rightarrow f(U')$ has a two sided inverse.  More precisely, Let $K_0$ be the norm of $L^{-1}$ and $K_1$ be the supremum of $|D^2f|$ over $U$. Choose $r$ small so that $K_0K_1r\leq1/2$ and $B_r\subset U$,
 then $f^{-1}$ is well-defined for $y\in F$ with $|y|\leq \frac{1}{2}K_0^{-1}r$, and for such $y$ we have $f^{-1}(y)\in B_r$. 

\end{lem}
\begin{proof}  This is standard. We write $f(x)=Lx+R(x)$, and define $T_y(x)= L^{-1}y-L^{-1}R(x)$.  We want to show $T_y$ is a contraction mapping from $B_r$ to itself, for $|y|\leq\frac{1}{2}K_0^{-1}r$.  This follows from easy computations. First we have, for $|y|\leq \frac{1}{2}K_0^{-1}r, |x|\leq r$ and $r$ such that $K_0K_1 r\leq 1/2$,
 \[|T_y (x)|\leq K_0|y|+K_0|Rx|\leq K_0|y|+K_0 K_1|x|^2\leq r.\]   Then we compute
\[\begin{split}|T_y (x_1)-T_y (x_2)|=& |L^{-1} (f(x_1)-f(x_2))-(x_1-x_2)|\\=&| L^{-1} Df(x')(x_1-x_2)-L^{-1}Df(0)(x_1-x_2)|\\ \leq& K_1K_0r |x_1-x_2|\leq \frac{1}{2}|x_1-x_2|,\end{split}\]where $x'$ is a point on the straight line joining $x_1$ and $x_2$.

\end{proof}

\begin{lem} \label{IFT1}
Suppose $E$ and $F$ are two Banach spaces. Let $U$ be an open neighborhood of the origin in $E$. Let $f$ be a smooth map from $U$ to $F$ with $f(0)=0$. Suppose $L=df|_{0}$ has a right inverse $\hat L: F\rightarrow E$.  Let $K_0$ be the norm of $\hat L$ and $K_1$ be the supremum of $|D^2f|$ over $U$. Choose $r$ so that $(K_0+1)K_1r\leq 1/2$ and $B_r\subset U$, and let $s=\frac{1}{2}(K_0+1)^{-1}r$, then we have a smooth map $P$ from $B_s\cap \Ker L$ to $E$, which is onto a neighborhood of $0$ in $f^{-1}(0)$,  so that  $dP|_0$ is the obvious inclusion of $\Ker L$ into $E$, and 
$$Im P\subset f^{-1}(0)\cap B_r.$$ Moreover we have a Hessian bound on $P$ in terms $K_0, K_1$.
\end{lem}
\begin{proof}
We  reduce to the situation of the previous lemma by considering the map $Q: U\rightarrow F\oplus Ker L; x\mapsto (f(x), x-\hat L L(x))$. Then $P(y)=Q^{-1}(0, y)$.  \end{proof}

\section{Compactness of Sasaki-Ricci solitons with positive curvature}\label{S-5}

\subsection{Geometric bound} \label{S-5-1}
In this section we prove compactness of Sasaki-Ricci solitons with positive transverse bisectional curvature.
First we need the following result for Sasaki-Ricci solitons.

\begin{prop}\label{P-5-1}
Let $(M_i, \xi_i, \eta_i, g_i)$ be a sequence of compact Sasaki Ricci solitons of dimension $2n+1$.  Assume that the volume $\Vol(M_i, g_i)$ is uniformly bounded  above by $V$ and the entropy $\mu_i(M_i, g_i)$ is uniformly bounded below by $\mu$. Then there are constants $D, R>0$ depending only on $V$, $\mu$ and $n$,  such that the transverse diameter $diam^T(M_i, g_i)\leq D$ and the transverse scalar curvature $|R^T(M_i, g_i)|\leq R$ for all $i$.  \end{prop}

This is a generalization of the corresponding results for K\"ahler-Ricci solitons, stated in \cite{Tian-Zhang}.  The proof is based on Perelman's results on K\"ahler-Ricci flow (see \cite{Sesum-Tian}).   In the Sasaki setting, one has to replace every quantity by its ``transverse" version. 
First we recall the definition of transverse distance on Sasaki manifolds.
\begin{defi}For $x, y\in M$, let $\cO_x, \cO_y$ denote the orbits of the Reeb vector field $\xi$ through $x, y$ respectively, then define
\[
d^T(x, y)=d(\overline{\cO_x}, \overline{\cO_y}). 
\]
\end{defi}

Then one has an obvious notion of ``transverse diameter". 
We note that there is uniform lower bound of the volume of a Sasaki-Ricci soliton give $\mu$-functional bounded below. 
\begin{lem}\label{P-vol}
For a Sasaki-Ricci soliton, we have
\begin{equation}\label{E-vol}
\Vol\geq \exp(\mu/4(n+1)-2n).
\end{equation}
\end{lem}
\begin{proof}
This is a rather standard fact for compact Ricci solitons and the proof here is almost identical. We give a sketch for completeness. 
Let $(M, g)$ be a Sasaki-Ricci soliton with normalized potential $f$. 
Recall the Sasaki-Ricci soliton equation implies
\[
R^T+\t f= 4n(n+1).
\]
Note that  $f$ is a critical point of $W$ functional and $\mu(g)=W(g, f)$. It then follows that
\[
2\t f-|\nabla f|^2+R^T+4(n+1)f=\mu. 
\]
Hence we can deduce the identity
\[
f=\frac{\mu+|\nabla f|^2+R^T}{4(n+1)}-2n.
\]
A straightforward maximum principle as in \cite{Ivey} (see Proposition 1) shows that $R^T\geq 0$. It then follows that $\min f\geq \frac{\mu}{4(n+1)}-2n$. The normalization condition $\int_M e^{-f}dv_g=1$ then implies $\log \Vol\geq \min f\geq \frac{\mu}{4(n+1)}-2n$. 

\end{proof}

In \cite{Collins} and \cite{He}, a uniform bound on the transverse diameter and transverse scalar curvature were derived  along a fixed Sasaki-Ricci flow, generalizing Perelman's results in Kahler-Ricci flow on Fano manifolds.  Proposition \ref{P-5-1}  follow directly from these results, by applying  the proof in \cite{Collins} and \cite{He} to any Sasaki-Ricci soliton $(M_i, \xi_i, \eta_i, g_i)$, noting that the volume of $(M, \xi_i, \eta_i, g_i)$ is uniformly bounded away from zero by Lemma \ref{P-vol}. 

\begin{rmk}The arguments in \cite{Collins} and \cite{He} are quite different when the Sasaki structure is irregular. The first author generalized Perelman's results to Sasaki-Ricci flow in quasiregular case and then use an approximation argument and the maximum principle to deal with the irregular case. Such an approximation argument works for irregular Sasaki-Ricci solitons since $\mu$ functional is bounded below given the volume properness as in Section 3.  Alternatively, one could also follow \cite{Collins} to obtain Proposition \ref{P-5-1} directly.  \end{rmk}

Proposition \ref{P-5-1} only provides transverse diameter bound. To obtain  compactness results for Sasaki-Ricci solitons, it is important to bound the diameter.  It is observed \cite{He} that the diameter is also bounded along a fixed Sasaki-Ricci flow; but such a bound depends on the Reeb foliation crucially and it is not applicable here. 
We fix a compact Sasaki-manifold $(M, \xi_0, \eta_0, g_0)$ with positive transverse bisectional curvature.  Let $(X, g_X, J)$ be its K\"ahler cone. Fix a maximal compact torus $\T$ in $\Aut(\xi_0, \eta_0, g_0)$, and denote by $\ft$ its Lie algebra. We denote by $\cM$ the space of all $\T$ invariant K\"ahler cone structures on $(X, J)$ that is a simple deformation of $g_X$. Then we have

\begin{prop}\label{P-5-3} For any K\"ahler cone structure  $(\xi, g_X)$ in $\cM$  with uniformly bounded volume,  there exists a closed orbit $\cO$ of $\xi$ on $M$ such that its length with respect to $g$ is uniformly bounded.
\end{prop}

\begin{proof} 
If $\dim \ft=1$, then all metrics in $\cM$ differ by a homothetic transformation, and the statement easily follows. 
Hence we can assume $2\leq \dim\ft\leq n+1$. By a small type I deformation if necessary, we may assume $(\xi_0, \eta_0, g_0)$ is quasi-regular.  For any $(M, \xi, \eta, g)\in \cM$, Lemma \ref{l-2-9} implies  $\eta_0(\xi)>0$, and then by Lemma \ref{l-2-6}, there is a type I deformation of $(\xi_0, \eta_0, g_0)$ with Reeb vector field $\xi$. Let $(\xi, \tilde{g}_X)$ be the K\"ahler cone metric, and $(\xi, \tilde \eta, \tilde g)$ the corresponding Sasaki structure on $M$. 
Then $( \xi, \tilde \eta, \tilde g)$ and $(\xi, \eta,  g)$ differ only by a transverse K\"ahler deformation.  Then for any closed orbit of $\xi$, it has the same length with respect to $g$ as  to $\tilde g$. Moreover, $(M, \xi, \eta, g)$ and $(M, \xi, \tilde \eta, \tilde g)$ have the same volume. So we only need to prove the statement for $(M, \xi, \tilde \eta, \tilde g)$.  To prove this,
first we note that the Reeb foliations of $\xi$ and $\xi_0$ always share some common orbits by Rukimbira \cite{Rukimbira95}. We sketch a proof here for completeness. 
Note that $\tilde \eta=\eta_0/\eta_0(\xi)$ on $M$. 
Let $p$ be a point of the maximum  of $\eta_0(\xi)$, then $d (\eta_0(\xi))=0$ at $p$. Since $\cL_{\xi_0}\eta_0(\xi)=0$, then $\eta_0(\xi)$ obtains the maximum  along the orbit $\cO_p$ of $\xi_0$ through $p$. Since along the orbit $\cO_p$,
\[
d\tilde \eta= (\eta_0(\xi))^{-1}d\eta_0,
\]
it follows that  $\xi$ and $\xi_0$ are proportional along $\cO_p$; hence $\cO_p$ is also a closed orbit of $\xi$. Now suppose the orbit $\cO_p$ has length $l$ with respect to $g_0$, then it has the length $ l\tilde \eta(\xi_0)=l /m$ with respect to the metric $\tilde g$, where $m$ is the value of $\eta_0(\xi)$ at $p$. Note that the volume of $(M, \tilde g)$ is given by
\[
\Vol(M, \tilde g)=(n! 2^n)^{-1}\int_M \tilde \eta \wedge (d\tilde \eta)^n=(n! 2^n)^{-1}\int_M (\eta_0(\xi))^{-(n+1)} \eta_0\wedge (d\eta_0)^n. 
\]
Since  $\eta_0(\xi)$ achieves a maximum at $p$, then
\[
\Vol(M, \tilde g)\geq m^{-(n+1)} \Vol (M, g_0). 
\] 
It then follows that $m^{-1}$, and hence the length of $\cO_p$ with respect to $\tilde g$, is uniformly bounded above, since the length $l$ is uniformly bounded above for the quasi-regular Sasaki structure $(\xi_0, \eta_0, g_0)$. This completes the proof. 
\end{proof}

Now we consider the compactness of a family of Sasaki-Ricci solitons. 

\begin{thm}\label{T-5-2}
Let $(M, \xi_i, \eta_i, g_i)$ be a sequence of Sasaki-Ricci solitons  in $\cM$ with positive (nonnegative) transverse bisectional curvature.  Assume that $\Vol(g_i)$ is uniformly bounded  above and $\mu(g_i)$ is uniformly bounded below. Then by passing to a subsequence it converges to a Sasaki-Ricci soliton $(M, \xi_{\infty}, \eta_{\infty}, g_\infty)$ with positive (nonnegative) transverse bisectional curvature in the smooth Cheeger-Gromov topology.
\end{thm}

\begin{proof}
By Proposition \ref{P-5-1}, we know the transverse diameter  and the transverse scalar curvature are uniformly bounded above. By the non-negativity of transverse bisectional curvature, the sectional curvature of $(M, g_i)$ is uniformly bounded.  By Proposition \ref{P-5-3}, the diameter of $(M, g_i)$ is then uniformly bounded. By \eqref{E-vol}, we have
\[
\Vol(g_i)\geq \exp(\mu( g_i) (4(n+1))^{-1}-2n).
\] Hence  by passing to a subsequence we get a limit  manifold
$(M, \xi_\infty, \eta_{\infty}, g_\infty)$ in the $C^{1,\alpha}$ topology. Then the Sasaki-Ricci soliton equation provides a uniform bound on all the $k$-th covariant derivatives of the Riemannian curvature tensor, so the convergence is in the smooth Cheeger-Gromov topology.  It is then clear that $(M, \xi_\infty, \eta_\infty, g_\infty)$ still satisfies the Sasaki-Ricci soliton equation, and the transverse bisectional curvature is nonnegative. If $(M, g_i)$ has positive transverse bisectional curvature, then by Proposition \ref{P-8-1} in the next section, $(M, \xi_\infty, \eta_\infty, g_\infty)$ still has positive transverse bisectional curvature.
\end{proof}

\subsection{Positivity}\label{S-8}
Let $(M, \xi_0, \eta_0, g_0)$ be a compact Sasaki manifold with positive transverse bisectional curvature, normalized by (\ref{normalization}). In \cite{He}, it is proved that the Sasaki-Ricci flow $(\xi(t), \eta(t), g(t))$ starting from $(\xi_0, \eta_0, g_0)$ has positive transverse bisectional curvature and uniformly bounded geometry for all time, and converges by sequence to a Sasaki-Ricci soliton $(M, \xi_\infty, \eta_\infty, g_\infty)$ with non-negative transverse bisectional curvature in the smooth Cheeger-Gromov topology. For the purpose of this paper we need to know the positivity also holds in the limit.

\begin{prop}\label{P-2-1}
 The limit Sasaki-Ricci soliton $(M, \xi_\infty, \eta_\infty, g_\infty)$ has positive transverse bisectional curvature. 
\end{prop} 

Such result follows from the K\"ahler setting \cite{CST} and it can be proved by showing the following two lemmata. 

\begin{lem}\label{l-8-2}
 Any limit Sasaki-Ricci soliton $(M, \xi_\infty, \eta_\infty, g_\infty)$ has positive transverse Ricci curvature.
\end{lem}

 \begin{proof} 
 The evolution equation for transverse Ricci curvature along the Sasaki-Ricci flow is given by
 \[
 \frac{\p Ric^T}{\p t}=\t Ric^T+Ric^T\cdot Rm^T-(Ric^T)^2
 \]
 where $(Ric^T \cdot Rm^T)_{i\bar j}=R^T_{l\bar k}R^T_{i\bar j k\bar l}$. 
 Using the maximum principle as in  K\"ahler setting (see Prop. 8.4 in \cite{He}, also Prop. 1 in \cite{Bando} and Prop. 1.1 in \cite{Mok}), it follows that nonnegative transverse bisectional curvature is preserved along the Sasaki-Ricci flow, and if the transverse Ricci curvature is positive at one point, it then becomes positive instantly. It is clear that $(M, \xi_\infty, \eta_\infty, g_\infty)$ has nonnegative transverse bisectional curvature, and that the transverse Ricci curvature of $(M, \xi_\infty, \eta_\infty, g_\infty)$ is positive at least at one point. Hence $(M, g_\infty)$ has positive transverse Ricci curvature everywhere. 
\end{proof}

Since the geometry is uniformly controlled along the flow, there is  indeed a uniform positive lower bound of transverse Ricci for any limit. Thus there are constants $C_2>C_1>0$, and a time $T_1$, such that 
\begin{equation}\label{e-8-2}
C_1 g^T_{t}\leq Ric^T_t \leq C_2 g^T_t, 
\end{equation}
for any $t>T_1$.

\begin{lem} \label{l-8-1}
There is a constant $\epsilon>0$, and a time $T_0>0$, such that the transverse bisectional curvature of $g(t)$ for $t\geq T_0$ has a uniform positive lower bound, i.e. for any two transverse $(1,0)$ tangent vectors $u$ and $v$ we have
\[
Rm^T_{i\bar jk\bar l}u^i u^{\bar j} v^k v^{\bar l}\geq \epsilon (g^T_{i\bar j}g^T_{k\bar l}+g^T_{i\bar l}g^T_{k\bar j})u^i u^{\bar j} v^k v^{\bar l}.
\]
Here the time $T_0$ may depend on the initial data, but the constant $\epsilon$ can be chosen to depend only on $C_1$ and $C_2$, not on the initial data.
\end{lem}

\begin{proof}

As in \cite{He}, the maximum principle arguments in K\"ahler setting can be carried over to transverse K\"ahler structure with slight modification.  
So this lemma follows similarly as in the K\"ahler case, see Lemma 6 in \cite{CST}. We sketch a proof here for completeness.  We know the evolution of the transverse curvature tensor is 
\begin{eqnarray}
\label{e-8-1} \frac{\p R^T_{i\bar{j}k\bar{l}}}{\p t}&=&\Delta
R^T_{i\bar{j}k\bar{l}}+R^T_{i\bar{j}p\bar{q}}R^T_{q\bar{p}k\bar{l}}-R^T_{i\bar{p}k\bar{q}}R^T_{p\bar{j}q\bar{l}}+
R^T_{i\bar{l}p\bar{q}}R^T_{q\bar{p}k\bar{j}}+R^T_{i\bar{j}k\bar{l}}\nonumber\\&&-\frac{1}{2}(R^T_{i\bar{p}}R^T_{p\bar{j}k\bar{l}}
+R^T_{p\bar{j}}R^T_{i\bar{p}k\bar{l}}+R^T_{k\bar{p}}R^T_{i\bar{j}p\bar{l}}+R^T_{p\bar{l}}R^T_{i\bar{j}k\bar{p}}).
\end{eqnarray}
We denote the operator $\square R^T_{i\bar jk\bar l}$ the right hand side of the above expression.
 As in \cite{CST}, we define a new transverse curvature type tensor
$$S(t)=Rm^T(t)-\l(t) g^T(t)* Ric^T(t), $$
where $\lambda(t)$ is a function to be determined, and 
$$(g^T*Ric^T)_{i\bar jk\bar l}=g^T_{i\bar j}Ric^T_{k\bar l}+g^T_{i\bar l} Ric^T_{k\bar j}+g^T_{k\bar l}Ric^T_{i\bar j}+g^T_{k\bar j}Ric^T_{i\bar l}.$$
Now a direct manipulation gives rise to 
\begin{eqnarray*}
\frac{\p S_{i\bar{j}k\bar{l}}}{\p t}&=&
\square
S_{i\bar{j}k\bar{l}}+\l(Ric^T*Ric^T)_{i\bar{j}k\bar{l}}+\l^2(g^T*(Ric^T\cdot
(g^T*Ric^T)))_{i\bar{j}k\bar{l}}\\\\&&-\l^2(Ric^T*((n+2)Ric^T+R^T
g^T)_{i\bar{j}k\bar{l}}-\l'(g^T*Ric^T)_{i\bar{j}k\bar{l}}\\\\&&+
\l^2[(g^T*Ric^T)_{i\bar{j}p\bar{q}}(g^T*Ric^T)_{q\bar{p}k\bar{l}}-
(g^T*Ric^T)_{i\bar{p}k\bar{q}}(g^T*Ric^T)_{p\bar{j}q\bar{l}}\\\\&&+(g^T*Ric^T)_{i\bar{l}p\bar{q}}(g^T*Ric^T)_{q\bar{p}k\bar{j}}].
\end{eqnarray*}
Here for a tensor $A=A_{i\bar{j}}$, and $B=B_{i\bar jk\bar l}$, we denote
$$(A\cdot B)_{i\bar j}=\sum_{k,l}A_{l\bar{k}}B_{i\bar jk \bar l}.$$
By (\ref{e-8-2}), we obtain for a constant $C_3$ depending only on $C_1$ and $C_2$ that
$$\frac{\p S_{i\bar{j}k\bar{l}}}{\p t}\geq \square
S_{i\bar{j}k\bar{l}}+(C_1^2\l-C_2\l'-C_3\l^2)g^T*g^T.$$
Let
$$\l(t)=\frac{C_4C_1^2e^{C_1^2t/C_2}}{1+C_4C_3e^{C_1^2t/C_2}},$$
then it satisfies $$C_1^2\l-C_2\l'-C_3\l^2=0.$$ 
Since $g(T_1)$ has positive transverse bisectional curvature, we can choose $C_4>0$ so small that $S(T_1)\geq 0$.
Then by Mok's maximum principle(see Proposition 8.5 in \cite{He}), we see for any $u$, $v$,  $$S(t)(u, \bar{u}, v, \bar{v})\geq 0$$ for all $t\geq T_1$. 
Since $\lim_{t\rightarrow\infty}\l(t)=C_1^2/C_3, $
we know that there is a time $T_0>T_1$, such that $$Rm^T_t(u, \bar{u}, v, \bar v)\geq \epsilon g^T_t*g^T_t(u, \bar{u} ,v, \bar{v}), $$ 
for $\epsilon=C_1^2/2C_3$ and all $u$, $v$. The uniformity is clear from the proof. 
\end{proof}

 \begin{prop}\label{P-8-1}
Let $(M, \xi_i, \eta_i, g_i)$ be a sequence of Sasaki-Ricci solitons on $M$ with positive transverse bisectional curvature. If the sequence converges to a limit $(M, \xi, \eta, g)$ in the Cheeger-Gromov sense, then $(M, \xi, \eta, g)$ also has positive transverse bisectional curvature. 
\end{prop}

\begin{proof}
Clearly $(M, \xi, \eta, g)$ is a Sasaki-Ricci soliton with nonnegative transverse bisectional curvature. By the defining equation, the transverse Ricci curvature is positive  at least at one point. Then the proof of Lemma \ref{l-8-2}  applies here to show that $(M, \xi, \eta, g)$ has  positive transverse Ricci curvature. Suppose $2c g^T\leq Ric^T\leq C g^T$ for positive constants $c$ and $C$. Then for $i$ large enough, we have $cg_i^T\leq Ric^T_i\leq 2C g_i^T$. Then applying Lemma \ref{l-8-1}  to $(M, \xi_i, \eta_i, g_i)$, we obtain a uniformly positive lower bound of the transverse bisectional curvature of $g_i$, depending only on $c$ and $C$. By taking limit we deduce Proposition \ref{P-8-1}. 
\end{proof}

\section{Discussions}\label{S-9}

There are several questions that seem to be interesting to the authors. 

\begin{enumerate}

\item  It is natural to ask whether one can always deform a Sasaki-Ricci soliton to a Sasaki-Einstein metric by a simple deformation. 
More precisely speaking,  we consider  simple deformation of Sasaki structures on $M$ by fixing a maximal torus $\T$  in $\Aut(X, J)$ with Lie algebra $\ft$. Let $\cR'_{KS}$ be the set of elements $\xi$ in $\cR'$ that is the Reeb vector field of a $\T$-invariant Sasaki-Ricci soliton.  Then by Theorem \ref{local deformation} we know $\cR'_{KS}$ is an open subset of $\cR'$, and  we would like to study the structure of  $\cR'_{KS}$. In particular we want to know whether it is true that $\cR^{'}_{KS}=\cR^{'}$ when $\cR^{'}_{KS}$ is not empty. One could try to understand this problem by a deformation method, as done in this paper. The main difficulty lies in the closedness property of the deformation. We have confirmed this in this paper when the soliton has positive transverse bisectional curvature, by proving geometric bounds. In general one would expect a more direct approach based on the study of space of all Sasaki metrics that arise as simple deformations.  Two related special cases have been considered previously. Futaki-Ono-Wang \cite{FOW}  proved that on a toric Sasaki manifold,  $\cR^{'}_{KS}=\cR^{'}$ and in particular, there exists a Sasaki-Einstein metric for the unique choice of Reeb vector field in $\cR^{'}$. Mabuchi-Nakagawa  \cite{MN} considered another special case which contains non toric examples.

\item An algebro-geometric counterpart of the main theorem, in particular, in the orbifold setting. It seems that a characterization of weighted projective spaces
         analogous to Kobayashi-Ochiai  would be very helpful.  
         
\item It is a question in Chen-Tian \cite{Chen-Tian2} to classify compact K\"ahler orbifolds with positive bisectional curvature. Corollary \ref{orbifold Frankel} gave an answer when the orbifold is polarized. This restriction comes from the fact that such orbifolds corresponds exactly to quasi-regular compact Sasaki manifolds. In general we may get a compact Sasaki orbifold. Given the fact(c.f. \cite{Chen-Tian2}) that a compact K\"ahler-Einstein orbifold with positive bisectional curvature is a global quotient of the projective space, it is tempting to expect that a compact Sasaki orbifold with positive transverse bisectional curvature is a finite quotient of a weighted Sasaki sphere. Then it would follow that a compact K\"ahler orbifold with positive bisectional curvature is bi-holomorphic to a finite quotient of a weighted projective space. A complete answer to this  would require a detailed analysis of the Sasaki-Ricci flow on orbifolds.

\end{enumerate}

\noindent Weiyong He\\
Department of Mathematics, University of Oregon, Eugene, Oregon, 97403\\
Email: whe@uoregon.edu\\

\noindent Song Sun\\
Department of Mathematics, Imperial College, London SW7 2AZ, U.K \\
Email: {s.sun@imperial.ac.uk}\\
New address: \\
Department of Mathematics, SUNY, Stony Brook, NY 11794, USA\\
Email: song.sun@stonybrook.edu

\end{document}